\documentclass[10pt,reqno]{amsart}

\usepackage[T1]{fontenc}
\usepackage[utf8]{inputenc}
\usepackage{lmodern}
\usepackage{amsmath,amssymb,mathtools,mathrsfs}
\usepackage[protrusion=true,expansion=false]{microtype}
\usepackage[colorlinks=true,linkcolor=blue,citecolor=blue,urlcolor=blue]{hyperref}
\hypersetup{
  pdftitle={Sharp Circular Sampling and Derivative Period Polynomials},
  pdfauthor={Seokho Jin},
  pdfsubject={Circular sampling, derivative period polynomials, simplicity, and interlacing}
}

\newcommand{\C}{\mathbb C}
\newcommand{\R}{\mathbb R}
\newcommand{\T}{\mathbb T}

\newcommand{\Zset}{\mathcal Z}

\newcommand{\J}{\mathcal J}
\newcommand{\ord}{\operatorname{ord}}
\newcommand{\re}{\operatorname{Re}}
\newcommand{\im}{\operatorname{Im}}

\newcommand{\eps}{\varepsilon}

\theoremstyle{plain}
\newtheorem{theorem}{Theorem}[section]
\newtheorem{proposition}[theorem]{Proposition}
\newtheorem{lemma}[theorem]{Lemma}
\newtheorem{corollary}[theorem]{Corollary}

\newtheorem{example}[theorem]{Example}
\theoremstyle{remark}
\newtheorem{remark}[theorem]{Remark}
\numberwithin{equation}{section}

\title[Sharp circular sampling and derivative period polynomials]
{Sharp Circular Sampling and Derivative Period Polynomials}
\author{Seokho Jin}
\address{Department of Mathematics, Chung-Ang University, 84 Heukseok-ro, Dongjak-gu, Seoul 06974, Republic of Korea}
  \email{archimed@cau.ac.kr}
\date{}

\begin{document}
\begin{abstract}
We determine the exact maximal reflected zero region that forces centered binomial samples of a balanced entire function to have all zeros on the unit circle.  In degree $d\ge2$, this region is
\[
 \Omega_d=\left\{a+ib:\ a^2-\frac{b^2}{d-1}\le\frac d4\right\}.
\]
The finite theorem is sharp already for a single reflected zero pair, and a phase-preserving canonical-product approximation extends it to balanced entire functions of order at most one.  De Bruijn strip contraction and projective Hermite--Kakeya--Obreschkoff theory then give the exact common-zero obstruction, simplicity, strict interlacing of consecutive derivative samples, and a monotone real-pencil root flow.

As an application, we prove the derivative-period-polynomial unit-circle theorem for completed $L$-functions of primitive holomorphic newforms, in every derivative order and for arbitrary level and nebentypus.  After the standard normalization, every zero of
\[
 \sum_{j=0}^{k-2}\binom{k-2}{j}\Lambda^{(m)}(f,j+1)z^j
\]
lies on the unit circle for every weight $k\ge4$, level, nebentypus, and derivative order $m\ge0$.  In particular, this proves the full-polynomial unit-circle conjecture of Diamantis and Rolen in its original level-one setting and extends it to arbitrary level and nebentypus.  The same source-side theorem also gives simplicity, strict interlacing, and, for each fixed derivative order, conductor-uniform quantitative localization in the weight aspect.
\end{abstract}

\maketitle

\section{Introduction}
This paper is about a sharp strip/orbit-to-circle sampling principle.  For each
degree $d$ and spacing $\delta$, we determine exactly which reflected zero orbits
may be allowed if the centered binomial sample
\[
 \sum_{j=0}^d\binom dj F\!\left(\delta\left(j-\frac d2\right)\right)z^j
\]
is to have all of its zeros on the unit circle for every balanced source.  The
answer is a precise reflected zero region; it is not merely a sufficient strip but
the maximal reflection-invariant source geometry for uniform circularity.  Thus the
main theorem is a source-side zero-geometry theorem before it is an arithmetic
application.

The period-polynomial results are obtained only after this deterministic theorem is
applied to completed modular $L$-functions.  The functional equation supplies the
balancing symmetry, while the Euler product and the reflected zero-free half-plane
place the centered zeros in the required strip.  Thus the circularity of derivative
period-polynomial zeros is not proved by estimating critical values one by one; it
is transported from completed zero geometry through a sharp sampling theorem.  The
same mechanism gives simplicity, strict interlacing, a real-pencil root flow, and
quantitative localization.

We first state the sampling theorem.  For $d\ge2$ and $\delta>0$, set
\begin{equation}
 B_{d,\delta}[F](z)=\sum_{j=0}^d\binom dj
 F\!\left(\delta\left(j-\frac d2\right)\right)z^j,
 \qquad B_d=B_{d,1}.
 \label{eq:intro-B}
\end{equation}
Write
\[
 \T=\{z:|z|=1\},
 \qquad
 \mathcal C_d=\{P\in\C[z]:\deg P=d,\ \Zset(P)\subseteq\T\},
\]
and let
\[
 (\J F)(s)=\overline{F(-\bar s)}.
\]
We call $F$ balanced if $F=\omega\J F$ for some $|\omega|=1$.  The non-fixed zeros
of a balanced function occur in reflection pairs $\{\rho,-\bar\rho\}$.  In
degree $d$ the correct source region for these pairs is
\[
 \Omega_{d,\delta}
 =\left\{a+ib:
 a^2-\frac{b^2}{d-1}\le\frac{d\delta^2}{4}\right\}.
\]

\begin{theorem}[Sharp circular sampling]
\label{thm:intro-sampling}
Let $d\ge2$ and $\delta>0$.
\begin{enumerate}
\item If $p$ is a nonzero balanced polynomial and
\[
 \Zset(p)\subseteq\Omega_{d,\delta},
\]
then $B_{d,\delta}[p]\in\mathcal C_d$.  Moreover, $\Omega_{d,\delta}$ is maximal
among reflection-invariant sets for which this implication holds uniformly for all
balanced polynomials.

\item Let $F\not\equiv0$ be a balanced entire function of order at most one.  If
\[
 \Zset(F)\subseteq S_h:=\{s:|\re s|\le h\},
 \qquad
 h\le\frac{\delta\sqrt d}{2},
\]
then, for every $m$ such that $F^{(m)}\not\equiv0$,
\[
 B_{d,\delta}[F^{(m)}]\in\mathcal C_d.
\]
The strip constant $\delta\sqrt d/2$ cannot be enlarged uniformly over this class.
\end{enumerate}
\end{theorem}

The theorem is exact in two complementary senses.  The finite polynomial statement is
governed by a single reflected pair: a zero on the imaginary axis gives a linear
unit-circle factor, while a non-fixed pair $\{\rho,-\bar\rho\}$ gives a
self-inversive quadratic, and a direct calculation shows that this quadratic has
both zeros on $\T$ exactly when $\rho\in\Omega_{d,\delta}$.  Schur--Szeg\H{o}
composition then multiplies the orbit factors without leaving the unit circle.  The
same two-point computation proves maximality, because any reflection-invariant
source set that reaches outside $\Omega_{d,\delta}$ already contains a quadratic
counterexample.  The passage from polynomials to entire functions is made by a
phase-preserving canonical-product approximation: reflected zero orbits are
truncated symmetrically, and the remaining exponential factor is approximated by
polynomials whose additional zeros lie on the imaginary axis.

The simplicity and interlacing assertions require a second, more rigid input.  For
$\zeta\in\T$ define
\[
 \mathscr T_{\zeta,\delta}g(s)
 =g\!\left(s-\frac\delta2\right)
  +\zeta g\!\left(s+\frac\delta2\right).
\]
The following theorem identifies the exact obstruction to a common zero of two
consecutive sampled derivatives.

\begin{theorem}[Strict circular sampling]
\label{thm:intro-strict}
Assume the hypotheses of Theorem~\ref{thm:intro-sampling}(2), with
\[
 h<\frac{\delta\sqrt d}{2}.
\]
Fix $m\ge0$ and suppose that
\[
 \mathscr T_{\zeta,\delta}^{\,d}F^{(m)}\not\equiv0
 \qquad\text{for every }\zeta\in\T.
\]
Then $B_{d,\delta}[F^{(m)}]$ and $B_{d,\delta}[F^{(m+1)}]$ have simple unit-circle
zeros and strictly cyclically interlace.  For every $t\in\R$, the zeros of
\[
 B_{d,\delta}[F^{(m)}]
 +itB_{d,\delta}[F^{(m+1)}]
\]
are simple and, after a fixed Cayley transformation, form monotone projective root
branches as $t$ varies; see Theorem~\ref{thm:strict-sampling}.
\end{theorem}

The obstruction is visible from the identities
\[
 B_{d,\delta}[E](\zeta)
   =\mathscr T_{\zeta,\delta}^{\,d}E(0),
 \qquad
 B_{d,\delta}[E'](\zeta)
   =\bigl(\mathscr T_{\zeta,\delta}^{\,d}E\bigr)'(0).
\]
Repeated de Bruijn strip contraction makes the zero strip thinner than the final
central-difference step, and the last nonzero central difference has only simple real
zeros.  Hence a common zero at $\zeta$ is equivalent to the identity
$\mathscr T_{\zeta,\delta}^{\,d}E\equiv0$.  Projective
Hermite--Kakeya--Obreschkoff theory then turns coprimeness into strict cyclic
interlacing and simplicity for the whole real pencil.

We now pass from the source-side theorem to arithmetic.  Let
\[
 f\in S_k^{\mathrm{new}}(\Gamma_0(N),\chi),\qquad k\ge4,
\]
be a normalized primitive holomorphic newform, and put
\[
 \Lambda(f,s)=\left(\frac{\sqrt N}{2\pi}\right)^s\Gamma(s)L(f,s).
\]
For $m\ge0$ define
\begin{equation}
 U_{f,m}(z)=\sum_{j=0}^{k-2}\binom{k-2}{j}\Lambda^{(m)}(f,j+1)z^j.
 \label{eq:intro-U}
\end{equation}
This is the unit-circle normalization of the full derivative period polynomial.

\begin{theorem}[Derivative period polynomials]
\label{thm:intro-newforms}
For every $m\ge0$, the polynomial $U_{f,m}$ has exactly $k-2$ zeros, all simple and
all on the unit circle.  The zeros of $U_{f,m}$ and $U_{f,m+1}$ strictly cyclically
interlace.  More generally, for every $t\in\R$, the polynomial
\[
 U_{f,m}+itU_{f,m+1}
\]
has only simple unit-circle zeros and, after a fixed Cayley transformation, its zeros
form monotone projective root branches as $t$ varies; see
Theorem~\ref{thm:strict-sampling}.
\end{theorem}

The bridge to Theorems~\ref{thm:intro-sampling} and~\ref{thm:intro-strict} is exact.
With
\[
 d=k-2,
 \qquad
 F_f(s)=\Lambda\!\left(f,\frac k2+s\right),
\]
one has
\begin{equation}
 U_{f,m}(z)=B_d[F_f^{(m)}](z).
 \label{eq:intro-bridge}
\end{equation}
The functional equation gives balancedness, and the Euler product together with the
reflected functional equation implies that all zeros of $F_f$ lie in $|\re s|\le1/2$.
Since $1/2<\sqrt{k-2}/2$ for $k\ge4$, Theorem~\ref{thm:intro-sampling} gives the
circular location.  A gamma-factor asymptotic at the right edge rules out the
strictness obstruction for every $m$ and every $\zeta\in\T$, completing the proof
of Theorem~\ref{thm:intro-newforms}.

The classical period-polynomial variable is recovered by the reciprocal rotation and
conductor rescaling
\[
 Q_{f,m}^{\mathrm{norm}}(z)=iz^{k-2}U_{f,m}\!\left(\frac1{iz}\right),
 \qquad
 R_{f,m}^{(N)}(X)=(-1)^{k-2}N^{-(k-1)/2}
 Q_{f,m}^{\mathrm{norm}}(\sqrt N X).
\]
Corollary~\ref{cor:conductor-rescaled} shows that $R_{f,m}^{(N)}$ has simple zeros
on $|X|=N^{-1/2}$ and inherits the same interlacing and root-flow conclusions.  For
$m=0$ this is the classical period polynomial
\[
 r_f(X)=\int_0^{i\infty}f(\tau)(\tau-X)^{k-2}\,d\tau.
\]
Thus Theorem~\ref{thm:intro-newforms} proves the full-polynomial unit-circle
problem posed by Diamantis and Rolen \cite{DiamantisRolen} in every derivative
order, extends it to arbitrary level and nebentypus, and at the same time
strengthens the conclusion to simplicity, strict interlacing, and a pencil flow; see
also \cite{DiamantisRolenSurvey} for background on the derivative
period-polynomial problem and its cohomological context.

The final main result makes the circular picture quantitative in the weight aspect,
uniformly in the conductor, nebentypus, and newform.  Let $\eps_f=e^{i\omega_f}$,
\[
 d=k-2,
 \qquad D=\frac d2,
 \qquad a_N=\frac{2\pi}{\sqrt N},
 \qquad
 \Gamma_{f,\mu}=\frac{\omega_f+\mu\pi}{2},
\]
and define
\begin{equation}
 \Phi_{N,f,\mu}(\theta)
 :=a_N\sin\theta-D\theta-\Gamma_{f,\mu}.
 \label{eq:intro-general-phase}
\end{equation}
Only $\Gamma_{f,\mu}$ modulo $\pi$ is relevant.  When $D>a_N$, let
$\mathcal A_{N,f,\mu}\subset\R/2\pi\mathbb Z$ be the $d$-point set defined by
\begin{equation}
 \cos\Phi_{N,f,\mu}(\theta)=0.
 \label{eq:intro-general-model-set}
\end{equation}

\begin{theorem}[Conductor-uniform quantitative localization]
\label{thm:intro-quantitative}
Fix $\mu\ge0$.  There exists $k_0(\mu)$ such that, for every normalized primitive
newform $f\in S_k^{\mathrm{new}}(\Gamma_0(N),\chi)$ with $k\ge k_0(\mu)$,
the roots $e^{i\vartheta_j}$ of $Q_{f,\mu}^{\mathrm{norm}}$ and the model angles
$\alpha_j\in\mathcal A_{N,f,\mu}$ can be cyclically labeled so that
\begin{equation}
 \operatorname{dist}_{\R/2\pi\mathbb Z}
 \bigl(\vartheta_j+\tfrac\pi2,\alpha_j\bigr)
 \ll_\mu \frac{a_N}{k^2}+\frac{2^{-k/3}}{k}.
 \label{eq:intro-general-localization}
\end{equation}
The implied constant is independent of $N$, $\chi$, and $f$.  Consequently, the
angular gaps for one derivative order are
\begin{equation}
 \frac{2\pi}{k-2}
 +O_\mu\!\left(\frac{a_N}{k^2}\right),
 \label{eq:intro-general-spacing}
\end{equation}
and, after merging the roots for orders $\mu$ and $\mu+1$, the two orders alternate
and every cross-order gap is
\begin{equation}
 \frac{\pi}{k-2}
 +O_\mu\!\left(\frac{a_N}{k^2}+\frac{2^{-k/3}}{k}\right)
 =\frac{\pi}{k-2}+O_\mu(k^{-2}).
 \label{eq:intro-general-half-step}
\end{equation}
\end{theorem}

Earlier circle theorems for $m=0$ were proved, in increasing generality, by Conrey,
Farmer, and Imamo\u{g}lu \cite{ConreyFarmerImamoglu}, by El-Guindy and
Raji \cite{ElGuindyRaji}, by Jin, Ma, Ono, and Soundararajan \cite{JinMaOnoSound}
for arbitrary level with trivial nebentypus, and by Liu, Park, and Song \cite{LiuParkSong}
for arbitrary nebentypus up to finitely many possible exceptional newforms.
First-derivative and related interlacing results were studied by Im and
Kim \cite{ImKim} and by Breland and collaborators \cite{BrelandEtAl}.  For
$\mu=0$, conductor-dependent phase models for period-polynomial zeros go back to
Jin--Ma--Ono--Soundararajan \cite{JinMaOnoSound}; Liu--Park--Song treated arbitrary
nebentypus, again with finitely many possible exceptional newforms.  Beyond
classical modular forms over $\mathbb Q$, Babei, Rolen, and Wagner \cite{BabeiRolenWagner}
defined a one-variable analogue of the period polynomial for parallel-weight Hilbert
modular eigenforms on the full Hilbert modular group and proved the corresponding
unit-circle theorem.  The arithmetic applications in this paper are stated for
classical newforms over $\mathbb Q$, but the sampling theorem is formulated at the
level of completed functional equations and is compatible with such one-variable
critical-value polynomials.  For background on the derivative period-polynomial
problem and its cohomological context, see also~\cite{DiamantisRolenSurvey}.
Diamantis--Rolen also formulated an odd-part assertion \cite{DiamantisRolen};
odd projection does not preserve the circular multiplier structure used here, and
is not addressed in this paper.

The earlier circle theorems are tied to the particular arithmetic structure and
asymptotics of modular critical values.  The present proof isolates the mechanism
before any newform is introduced.  Its finite part is governed by a single reflected
zero pair; its entire-function part is a phase-preserving approximation theorem; and
its strictness part identifies the exact obstruction to a common zero of consecutive
sampled derivatives.  This is the reason the final arithmetic statement is uniform
in the level, nebentypus, newform, and derivative order.

The proof combines several classical rigidity principles in a way tailored to the
sampling problem.  Grace--Szeg\H{o} composition~\cite{Grace,Szego} controls finite
reflected zero orbits.  The finite-difference Hermite--Poulain theorem~\cite{KatkovaTyaglovVishnyakova}
and de Bruijn's strip contraction~\cite{deBruijnTrig} control the passage from
strips to real zeros of central differences.  Projective HKO theory~\cite{Obreschkoff,RahmanSchmeisser}
then converts coprime circular pencils into strict cyclic interlacing and monotone
root branches.

The new source-side contributions are the sharp region $\Omega_{d,\delta}$ and its
universal maximality, the balance-preserving entire-function approximation, and the
exact common-zero criterion that converts circular location into derivative-order
dynamics.  The operator formulation is compatible with the Borcea--Br\"and\'en
algebraic-symbol framework~\cite{BorceaBranden} and the Br\"and\'en--Chasse
classification of strip preservers~\cite{BrandenChasse}, but the reflected-pair
source geometry and the resulting arithmetic application are the specific mechanisms
developed here.

Section~2 proves the finite orbit theorem, the maximality of $\Omega_{d,\delta}$,
and the associated circular multiplier statement.  Section~3 constructs
phase-preserving polynomial approximants and passes to entire functions.  Section~4
proves the common-zero obstruction, strict interlacing, and root flow.  Section~5
applies the framework to completed functional equations, Riemann's $\xi$-function,
and primitive newforms.  Section~\ref{sec:quantitative-localization} proves the
conductor-uniform quantitative localization and half-step theorem.

\section{Finite circular sampling}

We begin with polynomials.  The proof has two ingredients: the sampling map is multiplicative with respect to Schur--Szeg\H{o} composition, and the zeros of a balanced polynomial split into one- and two-point reflection orbits.  After setting up the composition law, we compute the sample of each orbit and read off the exact region $\Omega_{d,\delta}$.

\subsection{Schur--Szeg\H{o} composition and circular multipliers}

For fixed $d\ge1$, write degree-at-most-$d$ polynomials in the binomial basis as
\[
 P(z)=\sum_{j=0}^d\binom dj a_jz^j,
 \qquad
 Q(z)=\sum_{j=0}^d\binom dj b_jz^j.
\]
Their Schur--Szeg\H{o} composition is
\begin{equation}
 (P*_dQ)(z):=\sum_{j=0}^d\binom dj a_jb_jz^j.
 \label{eq:SS}
\end{equation}
It is commutative and associative.  We use the following circle form of the
Grace--Szeg\H{o} theorem; see the original papers~\cite{Grace,Szego} and the modern
formulation~\cite[Corollary~5.5.3(iii)]{RahmanSchmeisser}.  The
normalization in~\eqref{eq:SS} is the binomial normalization used there; no
reversal of coefficients is involved.

\begin{theorem}[Grace--Szeg\H{o}]
\label{thm:GraceSzego}
Let $P$ and $Q$ be polynomials of degree $d$. If every zero of $P$ lies on $|z|=r_1$ and every zero of $Q$ lies on $|z|=r_2$, then every zero of $P*_dQ$ lies on $|z|=r_1r_2$.
\end{theorem}

For a function $F$ defined at the centered lattice points, recall
\begin{equation}
 B_d[F](z)=\sum_{j=0}^d\binom djF\!\left(j-\frac d2\right)z^j.
 \label{eq:Bd-recall}
\end{equation}

\begin{lemma}[Multiplicativity]
\label{lem:multiplicativity}
For any $F$ and $G$ defined at the centered lattice points,
\begin{equation}
 B_d[FG]=B_d[F]*_dB_d[G].
 \label{eq:Bd-product}
\end{equation}
\end{lemma}

\begin{proof}
The $j$th binomial coefficient on each side is
$F(j-d/2)G(j-d/2)$.
\end{proof}

For a vector $\lambda=(\lambda_0,\ldots,\lambda_d)\in\C^{d+1}$, define
\[
 M_\lambda\!\left(\sum_{j=0}^d\binom dj a_jz^j\right)
 :=\sum_{j=0}^d\binom dj\lambda_ja_jz^j
\]
and its symbol
\[
 Q_\lambda(z):=\sum_{j=0}^d\binom dj\lambda_jz^j.
\]

\begin{proposition}[Finite circular multiplier criterion]
\label{prop:multiplier-criterion}
The following are equivalent:
\begin{enumerate}
\item[(i)] $M_\lambda(\mathcal C_d)\subseteq\mathcal C_d$;
\item[(ii)] $Q_\lambda\in\mathcal C_d$.
\end{enumerate}
\end{proposition}

\begin{proof}
If $Q_\lambda\in\mathcal C_d$ and $P\in\mathcal C_d$, then
$M_\lambda P=P*_dQ_\lambda\in\mathcal C_d$ by Theorem~\ref{thm:GraceSzego}.
Conversely, $(1+z)^d\in\mathcal C_d$ and
$M_\lambda(1+z)^d=Q_\lambda$.
\end{proof}

\begin{remark}[Classical status of the criterion]
\label{rem:classical-symbol}
Proposition~\ref{prop:multiplier-criterion} is the diagonal unit-circle-boundary
case of the algebraic-symbol theory of Borcea and Br\"and\'en~\cite{BorceaBranden}.
Indeed,
\[
 M_\lambda\bigl[(1+zw)^d\bigr]=Q_\lambda(zw).
\]
It also belongs to the classical P\'olya--Schur and Schur--Szeg\H{o} composition
framework; see~\cite{PolyaSchur,CravenCsordas,RahmanSchmeisser}.  We retain the
one-line proof because it fixes the normalization used below.  The contribution of
the finite sampling theorem is not this symbol criterion, but the determination of
source zero sets that force the symbol $B_d[F]$ to belong to $\mathcal C_d$.
\end{remark}

For a sampling function $F$, define
\begin{equation}
 M_{F,d}P(z):=\sum_{j=0}^d\binom djF\!\left(j-\frac d2\right)a_jz^j
 \quad\text{if}\quad
 P(z)=\sum_{j=0}^d\binom dj a_jz^j.
 \label{eq:MFd}
\end{equation}
Then
\begin{equation}
 M_{F,d}P=P*_dB_d[F].
 \label{eq:MFd-conv}
\end{equation}
Thus every theorem asserting $B_d[F]\in\mathcal C_d$ is automatically a theorem about a circular multiplier operator on the whole class $\mathcal C_d$.

The balance symmetry is carried by $B_d$ to self-inversiveness.

\begin{lemma}
\label{lem:self-inversive}
If $F=\omega\J F$, then
\begin{equation}
 B_d[F](z)=\omega z^d\overline{B_d[F](1/\bar z)}.
 \label{eq:self-inversive}
\end{equation}
\end{lemma}

\begin{proof}
The balance relation gives
\[
 F\!\left(j-\frac d2\right)
 =\omega\overline{F\!\left(\frac d2-j\right)}.
\]
Replace $j$ by $d-j$ and use $\binom d{d-j}=\binom dj$.
\end{proof}

Self-inversiveness identifies the unit circle as the natural symmetry curve, but it does not by itself force every zero onto that curve.

\subsection{Fixed reflection orbits}

Throughout the remainder of this section $d\ge2$.  We compute the sampling polynomial of each orbit of $\rho\mapsto-\bar\rho$.  The fixed points are $ib$, $b\in\R$.

\begin{lemma}
\label{lem:fixed-orbit}
For $b\in\R$,
\begin{equation}
 B_d[s-ib](z)
 =(1+z)^{d-1}\left[\left(\frac d2-ib\right)z-\left(\frac d2+ib\right)\right].
 \label{eq:fixed-orbit}
\end{equation}
Every zero of this polynomial lies on $\T$.
\end{lemma}

\begin{proof}
One has $B_d[1]=(1+z)^d$ and
\[
 B_d[s](z)
 =\sum_{j=0}^d\binom dj\left(j-\frac d2\right)z^j
 =\frac d2(z-1)(1+z)^{d-1}.
\]
This gives~\eqref{eq:fixed-orbit}. Apart from $-1$ of multiplicity $d-1$, the remaining zero is
\[
 \frac{d/2+ib}{d/2-ib}\in\T.
\]
\end{proof}

\subsection{Non-fixed reflection pairs and the exact region}
Let
\[
 \rho=a+ib,\qquad \rho^*:=-\bar\rho=-a+ib,
\]
and put
\begin{equation}
 q_\rho(s):=(s-\rho)(s-\rho^*)
 =s^2-2ibs-(a^2+b^2).
 \label{eq:q-rho}
\end{equation}

\begin{lemma}
\label{lem:quadratic-orbit}
One has
\begin{equation}
 B_d[q_\rho](z)
 =(1+z)^{d-2}\left(A_\rho z^2+B_\rho z+\overline{A_\rho}\right),
 \label{eq:Bq}
\end{equation}
where
\begin{align}
 A_\rho&=\frac{d^2}{4}-(a^2+b^2)-idb,
 \label{eq:A-rho}\\
 B_\rho&=\frac{d(2-d)}2-2(a^2+b^2)\in\R.
 \label{eq:B-rho}
\end{align}
Moreover,
\begin{equation}
 4|A_\rho|^2-B_\rho^2
 =d\bigl((d-1)(d-4a^2)+4b^2\bigr).
 \label{eq:discriminant-identity}
\end{equation}
\end{lemma}

\begin{proof}
Let $D_z=z\,d/dz$. Since
\[
 B_d[s^r]=(D_z-d/2)^r(1+z)^d,
\]
a direct calculation gives
\[
 B_d[s^2](z)
 =(1+z)^{d-2}\left[
 \frac{d^2}{4}z^2+\frac{d(2-d)}2z+\frac{d^2}{4}
 \right].
\]
Combining this with the formula for $B_d[s]$ proves~\eqref{eq:Bq}--\eqref{eq:B-rho}. Expanding $4A_\rho\overline{A_\rho}-B_\rho^2$ gives~\eqref{eq:discriminant-identity}.
\end{proof}

\begin{lemma}
\label{lem:selfreciprocal-quadratic}
Let $A\ne0$ and $B\in\R$. Both zeros of
$Az^2+Bz+\bar A$ lie on $\T$ if and only if $|B|\le2|A|$.
\end{lemma}

\begin{proof}
Write $A=|A|e^{i\theta}$ and set $z=e^{-i\theta}w$. The equation becomes
\[
 |A|w^2+Bw+|A|=0.
\]
Its roots have product one and real sum $-B/|A|$; they are unimodular exactly when the absolute value of that sum is at most two.
\end{proof}

Define the exact orbit region
\begin{equation}
 \Omega_d:=\left\{a+ib\in\C:\ a^2-\frac{b^2}{d-1}\le\frac d4\right\}.
 \label{eq:Omega}
\end{equation}

\begin{remark}[The degenerate degree-one case]
The restriction $d\ge2$ only excludes a degenerate endpoint case.  For the
scaled sample in degree one,
\[
 B_{1,\delta}[F](z)=F(-\delta/2)+F(\delta/2)z.
\]
Balancedness gives $|F(-\delta/2)|=|F(\delta/2)|$, so the single zero is
unimodular whenever the endpoint values are nonzero.  The maximal source set is
therefore not an $\Omega$-region, but the punctured set
$\C\setminus\{\pm\delta/2\}$.  The closed strip statement also fails on the
boundary: $F(s)=s^2-(\delta/2)^2$ is balanced and has zeros in
$|\re s|\le\delta/2$, but $B_{1,\delta}[F]\equiv0$.  Since the arithmetic
applications have $d=k-2\ge2$, we keep the general theory in the nondegenerate
range.
\end{remark}

\begin{proposition}[Exact two-orbit criterion]
\label{prop:two-orbit}
For a non-fixed orbit $\{\rho,-\bar\rho\}$, the following are equivalent:
\begin{enumerate}
\item[(i)] $B_d[q_\rho]\in\mathcal C_d$;
\item[(ii)] $(d-1)(d-4a^2)+4b^2\ge0$;
\item[(iii)] $\rho\in\Omega_d$.
\end{enumerate}
If the inequality is strict, the two zeros of the quadratic factor in~\eqref{eq:Bq} are distinct. Equality gives a double zero on $\T$.
\end{proposition}

\begin{proof}
If $A_\rho=0$, then $b=0$ and $|a|=d/2$. In this case $B_d[q_\rho]$ has degree less than $d$, while (ii) and (iii) both fail for $d\ge2$. Suppose now that $A_\rho\ne0$. Then the equivalence of (i) and (ii) follows from Lemmas~\ref{lem:quadratic-orbit} and~\ref{lem:selfreciprocal-quadratic}. Dividing (ii) by $4(d-1)$ gives (iii). The discriminant criterion in the proof of Lemma~\ref{lem:selfreciprocal-quadratic} gives the final statement.
\end{proof}

\begin{remark}[Finite free multiplicative convolution]
\label{rem:finite-free}
Let $\mathscr R_dP(x):=(-1)^dP(-x)$.  To fix the normalization, write
\[
 \widehat P(x)=\sum_{j=0}^d(-1)^{d-j}\binom dj a_jx^j,
 \qquad
 \widehat Q(x)=\sum_{j=0}^d(-1)^{d-j}\binom dj b_jx^j,
\]
and define their finite free multiplicative convolution by
\[
 (\widehat P\boxtimes_d\widehat Q)(x)
 :=\sum_{j=0}^d(-1)^{d-j}\binom dj a_jb_jx^j.
\]
This is the symmetric finite free multiplicative convolution, written in the
binomial normalization corresponding to Marcus--Spielman--Srivastava
\cite{MarcusSpielmanSrivastavaFiniteFree}.
Then the coefficient definitions give
\begin{equation}
 \mathscr R_d(P*_dQ)=(\mathscr R_dP)\boxtimes_d(\mathscr R_dQ).
 \label{eq:SS-finite-free}
\end{equation}
For $c\in\mathbb N_0$ and $\rho\in\C$,
\begin{equation}
 \mathscr R_d B_d[(s-\rho)^c](x)
 =L_d^*\!\left(x;-\frac d2-\rho,c\right),
 \label{eq:multiplicative-Laguerre-identification}
\end{equation}
where
\[
 L_d^*(x;b,c)=(x\partial_x+b)^c(x-1)^d
 =\sum_{j=0}^d(-1)^{d-j}\binom dj(j+b)^cx^j
\]
is Kabluchko's multiplicative Laguerre polynomial~\cite{KabluchkoMultiplicative}.  More precisely, for $F(s)=(s-\rho)^c$ the multiplier $M_{F,d}$ is the Euler operator $(z\partial_z-d/2-\rho)^c$.  When $\rho\in i\R$, its unit-circle preservation is exactly the unitary Euler-multiplier case of~\cite[Proposition~3.8]{KabluchkoMultiplicative}.  The paired construction used here begins with a non-fixed reflection pair: the two
individual Euler factors need not preserve unitary zeros, whereas their product does
so precisely on $\Omega_d$.  Theorem~\ref{thm:finite-orbit} then permits arbitrary
products of fixed and two-point reflection orbits.
\end{remark}

\begin{corollary}[Exact paired Euler-multiplier criterion]
\label{cor:paired-euler}
Let $D_z=z\,d/dz$ and $\beta\in\C$.  The operator
\[
 \mathcal E_{\beta,d}:=(D_z+\beta)(D_z-d-\bar\beta)
\]
maps $\mathcal C_d$ into itself if and only if
\begin{equation}
 \left(\re\beta+\frac d2\right)^2
 -\frac{(\im\beta)^2}{d-1}\le\frac d4.
 \label{eq:paired-euler-region}
\end{equation}
On the line $\re\beta=-d/2$ this contains the repeated unitary Euler multiplier.
Off that line neither first-order factor preserves $\mathcal C_d$ by itself, while
their paired product does so exactly in the region~\eqref{eq:paired-euler-region}.
\end{corollary}

\begin{proof}
Set $\rho=-d/2-\beta$.  Then
\[
 M_{(s-\rho)(s+\bar\rho),d}
 =(D_z+\beta)(D_z-d-\bar\beta).
\]
Proposition~\ref{prop:multiplier-criterion} and the exact two-orbit criterion give
preservation precisely when $\rho\in\Omega_d$, which is
\eqref{eq:paired-euler-region}.  For a single factor $D_z+\beta$, the sampled
symbol has its nontrivial zero on $\T$ exactly when $\re\beta=-d/2$; the same is
true for the reflected factor.
\end{proof}

\subsection{Products, universal maximality, and the sharp strip}

\begin{theorem}[Exact finite orbit-domain theorem]
\label{thm:finite-orbit}
Let $p\not\equiv0$ be a polynomial such that
\[
 p=\omega\J p,\qquad |\omega|=1,\qquad \Zset(p)\subseteq\Omega_d.
\]
Then $B_d[p]\in\mathcal C_d$. Equivalently,
\[
 M_{p,d}(\mathcal C_d)\subseteq\mathcal C_d.
\]
\end{theorem}

\begin{proof}
The zero multiset of $p$ splits into fixed orbits $\{ib\}$ and two-point orbits $\{\rho,-\bar\rho\}$. Up to a nonzero scalar, $p$ is a product of the corresponding linear factors $s-ib$ and quadratic factors $q_\rho$. By Lemma~\ref{lem:fixed-orbit} and Proposition~\ref{prop:two-orbit}, the sampling polynomial of every factor belongs to $\mathcal C_d$. The multiplicative identity~\eqref{eq:Bd-product} and Theorem~\ref{thm:GraceSzego} then show that $B_d[p]\in\mathcal C_d$. Its degree is $d$ because $d/2\notin\Omega_d$, so $p(d/2)\ne0$. The operator statement follows from Proposition~\ref{prop:multiplier-criterion}.
\end{proof}

The domain $\Omega_d$ is not merely sufficient; it is universally maximal for the reflection-paired implication in the following category.

\begin{theorem}[Universal setwise maximality of the orbit region]
\label{thm:maximality}
Let $E\subseteq\C$ be invariant under $\rho\mapsto-\bar\rho$. The following are equivalent:
\begin{enumerate}
\item[(i)] for every nonzero balanced polynomial $p$ with $\Zset(p)\subseteq E$, one has $B_d[p]\in\mathcal C_d$;
\item[(ii)] $E\subseteq\Omega_d$.
\end{enumerate}
\end{theorem}

\begin{proof}
The implication (ii)$\Rightarrow$(i) is Theorem~\ref{thm:finite-orbit}. Conversely, let $\rho\in E$. If $\rho$ is fixed, then $\rho\in i\R\subseteq\Omega_d$. Otherwise reflection invariance gives $-\bar\rho\in E$, and the balanced quadratic $q_\rho$ has both zeros in $E$. Assumption (i) and Proposition~\ref{prop:two-orbit} imply $\rho\in\Omega_d$.
\end{proof}

\begin{remark}[Scope of maximality]
The preceding theorem is a universal setwise statement. It does not assert that membership of every zero in $\Omega_d$ is necessary for a particular balanced polynomial to have a unit-circle-rooted image. For a specially chosen polynomial, contributions from distinct reflection orbits may cancel. What is proved is that no reflection-invariant set strictly larger than $\Omega_d$ can support the implication uniformly for all nonzero balanced polynomials with zeros in that set.
\end{remark}

The strip theorem is now a geometric corollary.

\begin{corollary}[Sharp finite strip theorem]
\label{cor:finite-strip}
For $d\ge2$ and $h\ge0$, the following are equivalent:
\begin{enumerate}
\item[(i)] every nonzero balanced polynomial $p$ with $\Zset(p)\subseteq S_h$ satisfies $B_d[p]\in\mathcal C_d$;
\item[(ii)] $h\le\sqrt d/2$.
\end{enumerate}
\end{corollary}

\begin{proof}
The inclusion $S_h\subseteq\Omega_d$ holds exactly when $h^2\le d/4$, since the left side of~\eqref{eq:Omega} is largest, for fixed $a$, at $b=0$. Apply Theorem~\ref{thm:maximality}. If $h>\sqrt d/2$, an explicit counterexample is
$p(s)=s^2-a^2$ with $\sqrt d/2<a\le h$.
\end{proof}

\begin{remark}
When $h=0$, the zeros lie on the imaginary axis.  After translating by $d/2$, this reduces to a classical Laguerre--Schur composition setting; see~\cite[Chapter~5]{RahmanSchmeisser}.  The passage to positive width is possible because a reflected pair is treated as one quadratic orbit rather than through separate one-zero annular bounds.
\end{remark}

\begin{remark}[Classical normalization and the square-root scale]
\label{rem:classical-normalization}
With $g(x)=p(x-d/2)$ one has
\[
 B_d[p](z)=\sum_{j=0}^d\binom dj g(j)z^j=g(D_z)(1+z)^d,
 \qquad D_z=z\frac d{dz},
\]
so the finite theorem belongs to the classical Laguerre--Schur composition setting.  De Bruijn's strip-contraction theorem~\cite[Theorem~5]{deBruijnTrig} exhibits the same threshold $\sqrt d/2$ for a polynomial formed from imaginary translates.  Its target and variables are different, however: it does not contain the exact reflected-pair region~\eqref{eq:Omega} or the universal setwise maximality of Theorem~\ref{thm:maximality}.
\end{remark}

\section{Entire-function extension}

The finite theorem does not yet apply to completed $L$-functions, which are entire rather than polynomial.  We therefore need polynomial approximants that retain the two features used in Section~2: the location of the zeros and the exact balance phase.  Ordinary Taylor truncations need not preserve either feature, so we work instead with canonical products adapted to the reflection symmetry.

\subsection{Balance-preserving canonical-product approximation}

We first formulate the approximation theorem for an arbitrary reflection-invariant zero region.

\begin{theorem}[Balanced approximation in a reflection-invariant set]
\label{thm:balanced-approx}
Let $E\subseteq\C$ be closed, invariant under $\rho\mapsto-\bar\rho$, and suppose that $i\R\subseteq E$. Let $F\not\equiv0$ be an entire function of order at most one satisfying
\[
 \Zset(F)\subseteq E,
 \qquad F=\omega\J F,
 \qquad |\omega|=1.
\]
Then there are polynomials $p_N$ such that
\begin{align}
 p_N&\longrightarrow F &&\text{locally uniformly on $\C$},
 \label{eq:approx-conv}\\
 \Zset(p_N)&\subseteq E,
 \label{eq:approx-zero}\\
 p_N&=\omega\J p_N.
 \label{eq:approx-phase}
\end{align}
\end{theorem}

\begin{proof}
Write $\omega=e^{i\theta}$ and choose $c=e^{-i\theta/2}$, so that
$\bar c=\omega c$.  Then $G:=cF$ satisfies
\begin{equation}
 \J G=G.
 \label{eq:J-G}
\end{equation}
It is therefore enough to construct $\J$-invariant polynomial approximants to $G$;
we will multiply them by $c^{-1}$ at the end.

Let $r=\ord_0G$.  Hadamard factorization for an entire function of order at most
one (see, for example, Levin~\cite[Chapter~I]{Levin}) gives
\begin{equation}
 G(s)=C_0e^{\alpha s}s^r\prod_{\rho\ne0}E_1(s/\rho),
 \qquad E_1(w):=(1-w)e^w,
 \label{eq:Hadamard}
\end{equation}
where the zeros are listed with multiplicity and
$\sum_{\rho\ne0}|\rho|^{-2}<\infty$.  We spell out why the factors may be
regrouped according to the reflection.  If $K\subset\C$ is compact, then for all
sufficiently large $|\rho|$ one may choose the principal logarithm near $1$ and
has
\[
 \sup_{s\in K}|\log E_1(s/\rho)|\le C_K|\rho|^{-2}.
\]
Thus the logarithmic tail is absolutely and uniformly convergent on $K$.  After
separating finitely many initial zeros, every permutation of the remaining factors,
and in particular every grouping into finite reflection orbits, has the same locally
uniform product.

The zero multiset is invariant under $\rho\mapsto-\bar\rho$.  Enumerate its
nonzero reflection orbits by $\mathcal O_1,\mathcal O_2,\ldots$.  If there are
only finitely many nonzero orbits, extend the list by empty orbits.  With the empty
product interpreted as $1$, set
\[
 \Phi_\nu(s):=\prod_{\rho\in\mathcal O_\nu}E_1(s/\rho).
\]
Because $E_1$ has real coefficients,
\[
 \J\bigl(E_1(s/\rho)\bigr)=E_1\bigl(s/(-\bar\rho)\bigr),
\]
so each $\Phi_\nu$ is $\J$-invariant.  The factor $is$ is also
$\J$-invariant.  Absorbing $i^{-r}$ into the constant, we may write
\begin{equation}
 G(s)=Ce^{\alpha s}(is)^r\prod_{\nu=1}^{\infty}\Phi_\nu(s).
 \label{eq:orbit-product}
\end{equation}
The product is locally uniformly convergent by the preceding normal-convergence
argument.  The finite-orbit and zero-free cases are covered by the empty-orbit
convention above; in the zero-free case this simply means $r=0$ and every
$\Phi_\nu$ is the empty product $1$.

Put
\[
 P(s):=(is)^r\prod_{\nu\ge1}\Phi_\nu(s).
\]
Then $\J P=P$.  On the open set on which $P$ does not vanish, the identity
$\J G=G$ and~\eqref{eq:orbit-product} give
\[
 Ce^{\alpha s}=\bar C e^{-\bar\alpha s}.
\]
Both sides are entire, so the identity holds on all of $\C$.  Evaluating it at
$s=0$ and then differentiating at $s=0$ shows that
$C=\bar C\ne0$ and $\alpha=-\bar\alpha$.  Hence $C\in\R\setminus\{0\}$ and
$\alpha=i\tau$ for some $\tau\in\R$, and
\begin{equation}
 G(s)=Ce^{i\tau s}(is)^r\prod_{\nu=1}^{\infty}\Phi_\nu(s).
 \label{eq:G-final-factorization}
\end{equation}

For $N\ge1$, define
\begin{equation}
 R_N(s):=(is)^r
 \prod_{\nu=1}^N\prod_{\rho\in\mathcal O_\nu}
 \left(1-\frac s\rho\right).
 \label{eq:RN}
\end{equation}
Each orbitwise linear product is $\J$-invariant, so $\J R_N=R_N$; all zeros of
$R_N$ are zeros of $G$ and therefore belong to $E$.  Moreover,
\[
 (is)^r\prod_{\nu=1}^N\Phi_\nu(s)
 =R_N(s)e^{\sigma_Ns},
 \qquad
 \sigma_N:=\sum_{\nu=1}^N\sum_{\rho\in\mathcal O_\nu}\frac1\rho.
\]
Every orbit contributes a purely imaginary number.  For a two-point orbit,
\[
 \frac1\rho+\frac1{-\bar\rho}
 =\frac1\rho-\frac1{\bar\rho}\in i\R,
\]
and for a fixed point $\rho=ib$, one has $1/\rho\in i\R$.  Thus
$\sigma_N=i\eta_N$ with $\eta_N\in\R$.  Set
\[
 t_N:=\tau+\eta_N,
 \qquad
 G_N(s):=CR_N(s)e^{it_Ns}.
\]
This is exactly the $N$th orbitwise truncation of~\eqref{eq:G-final-factorization}:
\[
 G_N(s)=Ce^{i\tau s}(is)^r\prod_{\nu=1}^N\Phi_\nu(s).
\]
Consequently $G_N\to G$ locally uniformly.  Notice that no convergence of the
real sequence $(t_N)$ is asserted or needed.

For each fixed real $t$,
\begin{equation}
 \left(1+\frac{its}{M}\right)^M\longrightarrow e^{its}
 \label{eq:exp-approx}
\end{equation}
locally uniformly as $M\to\infty$.  We use this separately for each value $t_N$.
Let
\[
 A_N:=\max_{|s|\le N}|CR_N(s)|
\]
and choose $M_N$ so large that
\begin{equation}
 \max_{|s|\le N}
 \left|\left(1+\frac{it_Ns}{M_N}\right)^{M_N}-e^{it_Ns}\right|
 <\frac1{N(1+A_N)}.
 \label{eq:diagonal-choice}
\end{equation}
When $t_N=0$, the parenthesized factor is interpreted as $1$.  Define
\begin{equation}
 q_N(s):=CR_N(s)
 \left(1+\frac{it_Ns}{M_N}\right)^{M_N}.
 \label{eq:qN}
\end{equation}
If $t_N\ne0$, the only additional zero is $iM_N/t_N$, with multiplicity $M_N$;
it lies on $i\R\subseteq E$.  The extra factor is $\J$-invariant, and hence
$\J q_N=q_N$.  By~\eqref{eq:diagonal-choice},
\[
 \max_{|s|\le N}|q_N(s)-G_N(s)|<\frac1N.
\]
For a fixed compact set $K$, this bound applies for all sufficiently large $N$;
combined with $G_N\to G$, it proves $q_N\to G$ locally uniformly.

Finally set $p_N=c^{-1}q_N$.  Then $p_N\to F$, its zero set is unchanged, and
\[
 \omega\J p_N
 =\omega\bar c^{-1}\J q_N
 =c^{-1}q_N=p_N,
\]
because $\bar c=\omega c$.
\end{proof}

\begin{remark}[Relation to classical strip approximation]
For a closed strip, polynomial approximation preserving the strip of zeros already
appears in de Bruijn's work~\cite{deBruijnTrig}.  The point of
Theorem~\ref{thm:balanced-approx} is not the existence of strip-preserving
approximants alone: it treats an arbitrary reflection-invariant set and simultaneously
preserves the exact anti-linear phase $F=\omega\J F$.  The orbitwise regrouping and
the diagonal approximation of the residual factor $e^{it_Ns}$ are included to make
those two additional requirements explicit.
\end{remark}

\begin{corollary}[Differentiation in convex regions]
\label{cor:differentiation}
Assume in addition that $E$ is convex. For every $m\ge0$ with $F^{(m)}\not\equiv0$, there are polynomials $q_N$ such that
\[
 q_N\to F^{(m)}\quad\text{locally uniformly},
 \qquad \Zset(q_N)\subseteq E,
 \qquad q_N=(-1)^m\omega\J q_N.
\]
In particular, $\Zset(F^{(m)})\subseteq E$.
\end{corollary}

\begin{proof}
By Cauchy's integral formula, locally uniform convergence is preserved under every
fixed number of derivatives.  Thus $p_N^{(m)}\to F^{(m)}$ locally uniformly.  Since
$F^{(m)}\not\equiv0$, the derivatives $p_N^{(m)}$ are nonzero for all sufficiently
large $N$; discard the remaining finitely many terms and put $q_N=p_N^{(m)}$.
The identity
\[
 D^m\J=(-1)^m\J D^m
\]
gives $q_N=(-1)^m\omega\J q_N$.  By Gauss--Lucas, every zero of
$p_N^{(m)}$ lies in the convex hull of the zeros of $p_N$, which is contained in
$E$ because $E$ is convex.

To obtain the final assertion, let $K$ be a closed disk disjoint from $E$.  Each
$q_N$ is zero-free on $K$.  Hurwitz's theorem says that the locally uniform limit is
either zero-free in the interior of $K$ or identically zero there.  The second
alternative would force $F^{(m)}\equiv0$ on $\C$, contrary to the hypothesis.
Hence $F^{(m)}$ has no zero outside $E$.
\end{proof}

\subsection{Exact and strip sampling}

\begin{theorem}[Exact entire orbit-domain theorem]
\label{thm:entire-orbit}
Let $d\ge2$, and let $F\not\equiv0$ be entire of order at most one and suppose
\[
 F=\omega\J F,
 \qquad |\omega|=1,
 \qquad \Zset(F)\subseteq\Omega_d.
\]
Then
\[
 B_d[F]\in\mathcal C_d,
 \qquad M_{F,d}(\mathcal C_d)\subseteq\mathcal C_d.
\]
\end{theorem}

\begin{proof}
The set $\Omega_d$ is closed, reflection invariant, and contains $i\R$. Apply Theorem~\ref{thm:balanced-approx} with the closed set $\Omega_d$. For its approximants $p_N$, Theorem~\ref{thm:finite-orbit} gives $B_d[p_N]\in\mathcal C_d$. Since $B_d$ uses only $d+1$ fixed values, $B_d[p_N]\to B_d[F]$ coefficientwise, hence locally uniformly. The endpoints $\pm d/2$ do not lie in $\Omega_d$, so $F(\pm d/2)\ne0$ and the limit has degree $d$. If $B_d[F]$ had an off-circle zero, Hurwitz's theorem on a small disk disjoint from $\T$ would force $B_d[p_N]$ to have a zero there for all large $N$, a contradiction. Thus $B_d[F]\in\mathcal C_d$, and the operator statement follows from Proposition~\ref{prop:multiplier-criterion}.
\end{proof}

The universal region does not become larger if polynomial test functions are excluded.

\begin{theorem}[Sharpness within exact order one]
\label{thm:exact-order-sharpness}
Let $d\ge2$, and let $E\subseteq\C$ be invariant under $\rho\mapsto-\bar\rho$.  The following are
equivalent:
\begin{enumerate}
\item[(i)] for every balanced nonpolynomial entire function $F$ of order exactly one
with $\Zset(F)\subseteq E$, one has $B_d[F]\in\mathcal C_d$;
\item[(ii)] $E\subseteq\Omega_d$.
\end{enumerate}
\end{theorem}

\begin{proof}
If $E\subseteq\Omega_d$, then (i) follows from
Theorem~\ref{thm:entire-orbit}.  Conversely, suppose that
$\rho\in E\setminus\Omega_d$.  Since every fixed point of the reflection lies on
$i\R\subseteq\Omega_d$, the orbit of $\rho$ has two points and
$q_\rho(s)=(s-\rho)(s+\bar\rho)$ is $\J$-invariant.  For any
$t\in\R\setminus\{0\}$, put
\[
 F_{\rho,t}(s):=e^{its}q_\rho(s).
\]
Since $t$ is real, $e^{its}$ is $\J$-invariant, and since $t\ne0$ it has
positive exponential type.  Hence $F_{\rho,t}$ is a nonpolynomial
$\J$-invariant entire function of order exactly one, and its zeros lie in $E$.
Directly from the definition of $B_d$,
\begin{equation}
 B_d[F_{\rho,t}](z)
 =e^{-itd/2}B_d[q_\rho](e^{it}z).
 \label{eq:exponential-rotation}
\end{equation}
By Proposition~\ref{prop:two-orbit}, $B_d[q_\rho]\notin\mathcal C_d$; multiplication
by a nonzero scalar and rotation of the variable cannot repair either an off-circle
zero or a degree drop.  Hence (i) fails.
\end{proof}

For derivatives we use the maximal contained strip, since $\Omega_d$ itself is not convex.

\begin{theorem}[Derivative strip theorem]
\label{thm:derivative-strip}
Let $d\ge2$, $0\le h\le\sqrt d/2$, and let $F\not\equiv0$ be entire of order at most one with
\[
 \Zset(F)\subseteq S_h,
 \qquad F=\omega\J F,
 \qquad |\omega|=1.
\]
Then, for every $m\ge0$ with $F^{(m)}\not\equiv0$,
\begin{equation}
 B_d[F^{(m)}]\in\mathcal C_d,
 \qquad M_{F^{(m)},d}(\mathcal C_d)\subseteq\mathcal C_d.
 \label{eq:derivative-strip-conclusion}
\end{equation}
The half-width $\sqrt d/2$ is optimal for a theorem uniform over this class.
\end{theorem}

\begin{proof}
By Corollary~\ref{cor:differentiation} with the convex set $S_h$, each $F^{(m)}$ is balanced with phase $(-1)^m\omega$, has order at most one, and has all zeros in $S_h\subseteq\Omega_d$. Apply Theorem~\ref{thm:entire-orbit}. Sharpness follows from the quadratic counterexamples in Corollary~\ref{cor:finite-strip}.
\end{proof}

\subsection{Scaled sampling}
For $\delta>0$, define
\begin{equation}
 B_{d,\delta}[F](z):=\sum_{j=0}^d\binom dj
 F\!\left(\delta\left(j-\frac d2\right)\right)z^j.
 \label{eq:scaled-B}
\end{equation}
The exact scaled orbit region is
\begin{equation}
 \Omega_{d,\delta}:=\delta\Omega_d
 =\left\{a+ib:\ a^2-\frac{b^2}{d-1}\le\frac{d\delta^2}{4}\right\}.
 \label{eq:scaled-Omega}
\end{equation}

\begin{corollary}[Scaled exact and strip theorems]
\label{cor:scaled}
Let $d\ge2$, $\delta>0$, and let $F\not\equiv0$ be balanced and entire of order at most one. If $\Zset(F)\subseteq\Omega_{d,\delta}$, then $B_{d,\delta}[F]\in\mathcal C_d$ and its sampled diagonal operator preserves $\mathcal C_d$. If instead $\Zset(F)\subseteq S_h$ with
\[
 h\le\frac{\delta\sqrt d}{2},
\]
then the same conclusions hold for every nonzero derivative $F^{(m)}$. The strip constant is sharp.
\end{corollary}

\begin{proof}
Apply Theorems~\ref{thm:entire-orbit} and~\ref{thm:derivative-strip} to $G(s):=F(\delta s)$. Since
$G^{(m)}(s)=\delta^mF^{(m)}(\delta s)$, one has
\[
 B_d[G^{(m)}]=\delta^mB_{d,\delta}[F^{(m)}],
\]
and the nonzero factor $\delta^m$ does not affect the conclusion.
\end{proof}

\begin{remark}[Finite Fourier interpretation]
\label{rem:finite-fourier}
For
\[
 \Phi_d(P)(t):=e^{-idt/2}P(e^{it}),
 \qquad
 \mathcal E_d=\operatorname{span}\{e^{i(j-d/2)t}:0\le j\le d\},
\]
the sampled multiplier is unitarily equivalent to the finite spectral multiplier
\[
 \Phi_d(M_{F,d}P)=F(-i\partial_t)\Phi_d(P),
\]
where $F(-i\partial_t)e^{i(j-d/2)t}=F(j-d/2)e^{i(j-d/2)t}$.  Hence the entire sampling theorem may also be read as a finite trigonometric Hermite--Poulain theorem on the centered Fourier space.  This viewpoint is related in spirit to, but distinct from, the finite-difference operators studied in~\cite{KatkovaTyaglovVishnyakova}.
\end{remark}

\section{Simplicity, interlacing, and cyclic root flow}
\label{sec:global-interlacing}

We now strengthen the unit-circle theorem to simplicity and strict interlacing.
Two zero-location ingredients are classical: the projective
Hermite--Kakeya--Obreschkoff theorem for real pencils and de Bruijn's contraction of a
zero strip under imaginary translation.  The new step is to combine them with the
sampling identities~\eqref{eq:T-sampling-identities}: the contraction argument shows
that the only possible common-zero obstruction for two consecutive sampled
derivatives is an identically vanishing finite-difference iterate.  Excluding exactly
that obstruction produces a circular hyperbolic pencil with strict interlacing and
monotone root flow.

\subsection{Derivative pencils and circular Obreschkoff theory}

\begin{lemma}[Derivative pencils preserve a vertical strip]
\label{lem:derivative-pencil-strip}
Let $E\not\equiv0$ be an entire function of order at most one satisfying
\[
 \Zset(E)\subseteq S_h,
 \qquad E=\nu\J E,
 \qquad |\nu|=1.
\]
For $t\in\R$, suppose that $E+itE'\not\equiv0$.  Then
\[
 \Zset(E+itE')\subseteq S_h,
 \qquad E+itE'=\nu\J(E+itE').
\]
\end{lemma}

\begin{proof}
Differentiating $E=\nu\J E$ gives $E'=-\nu\J(E')$, and hence
$iE'=\nu\J(iE')$.  This proves the phase assertion.

For the zero location, apply Theorem~\ref{thm:balanced-approx} with the convex set
$S_h$ and write $p_n\to E$ for the balanced polynomial approximants.  If one of
these approximants is constant, then $p_n+itp_n'=p_n$ is zero-free, so there is
nothing to prove for that approximant.  Thus, in the logarithmic-derivative argument,
we may assume that $p_n$ is nonconstant.  If $t\ne0$ and $z\notin S_h$ were a zero
of $p_n+itp_n'$, then $p_n(z)\ne0$ and
\[
 \frac{p_n'(z)}{p_n(z)}=\frac{i}{t}.
\]
Writing the zeros of $p_n$ as $\rho$ gives
\[
 \re\frac{p_n'(z)}{p_n(z)}
 =\sum_\rho\frac{\re z-\re\rho}{|z-\rho|^2},
\]
which is positive for $\re z>h$ and negative for $\re z<-h$.  This contradicts the
purely imaginary value above.  Hence $p_n+itp_n'$ has all zeros in $S_h$.
Locally uniform convergence and Hurwitz's theorem prove the assertion for $E$.
The case $t=0$ is immediate.
\end{proof}

We use a projective form of Obreschkoff's theorem.  Two degree-$d$ multisets on
$\T$ weakly cyclically interlace if, after a Cayley transformation, they weakly
interlace on $\R\cup\{\infty\}$.  Strict cyclic interlacing means that both multisets
are simple and disjoint and that their points alternate around the circle.

\begin{lemma}[Projective circular Obreschkoff lemma]
\label{lem:circular-obreschkoff}
Let $A$ and $B$ be degree-$d$ polynomials with the same self-inversive phase:
\[
 A(z)=\nu z^d\overline{A(1/\bar z)},
 \qquad
 B(z)=\nu z^d\overline{B(1/\bar z)},
 \qquad |\nu|=1.
\]
Assume that, for every $(a,b)\in\R^2\setminus\{(0,0)\}$, the polynomial
$aA+bB$ is nonzero, has degree $d$, and has all $d$ zeros on $\T$.  Then the
zeros of $A$ and $B$ weakly cyclically interlace.  If $A$ and $B$ are coprime,
then both have simple zeros, their zeros strictly cyclically interlace, and every
nonzero real pencil member has only simple zeros on $\T$.
\end{lemma}

\begin{proof}
Consider the Cayley parametrization
\[
 \phi([X:Y])=\frac{X-iY}{X+iY},
\]
which is a bijection from $\mathbb{RP}^1$ onto $\T$.  Choose one square root
$\nu^{1/2}$ and homogenize by
\[
 \widetilde A(X,Y)
 :=\nu^{-1/2}(X+iY)^d
 A\!\left(\frac{X-iY}{X+iY}\right),
\]
with the analogous definition for $\widetilde B$.  The displayed expression is a
binary form of degree $d$; the formula extends polynomially across $X+iY=0$.
For real $X,Y$, the self-inversive relation gives
\[
 \overline{(X+iY)^dA(\phi([X:Y]))}
 =\nu^{-1}(X+iY)^dA(\phi([X:Y])).
\]
It follows that $\widetilde A$ is real-valued on $\R^2$, hence has real
coefficients; the same holds for $\widetilde B$.  Multiplicities are preserved, and
zeros on $\T$ correspond exactly to real projective zeros of these binary forms.  In
particular, a zero represented by the point at infinity in one affine chart is treated
on the same footing as every other projective zero.

The hypothesis says that every nonzero form
$a\widetilde A+b\widetilde B$ is projectively hyperbolic of degree $d$.  The
projective Hermite--Kakeya--Obreschkoff theorem
\cite{Obreschkoff} (see also~\cite[Theorem~6.3.8]{RahmanSchmeisser}) therefore gives weak projective
interlacing, which is weak cyclic interlacing after applying $\phi$.

Assume now that $A$ and $B$ are coprime.  Then the binary forms
$\widetilde A$ and $\widetilde B$ have no common projective zero.  We show that
no nonzero pencil member has a multiple projective zero.  Suppose that a pencil
member $C$ has a zero $x_0\in\mathbb{RP}^1$ of multiplicity $q\ge2$.  Choose an
independent pencil member $D$ with $D(x_0)\ne0$.  Such a choice is possible because
the evaluation functional at $x_0$ is not identically zero on the two-dimensional
pencil; otherwise $x_0$ would be a common zero of $\widetilde A$ and
$\widetilde B$.  Select an affine chart containing $x_0$---using the reciprocal
chart if $x_0$ is the point at infinity---and write
\[
 h(x):=\frac{C(x)}{D(x)}
      =c(x-x_0)^q+O\bigl((x-x_0)^{q+1}\bigr),
 \qquad c\ne0.
\]
Choose a small closed disk around $x_0$ on which $D$ has no zero and on whose
boundary $C$ has no zero.  If $q$ is even, then $h$ has the sign of $c$ on both
sides of $x_0$ in a sufficiently small punctured real neighborhood.  For one sign
of sufficiently small real $\varepsilon$, the equation $h+\varepsilon=0$ has no
real solution there.  If $q$ is odd, then $q\ge3$ and
\[
 h'(x)=cq(x-x_0)^{q-1}\bigl(1+O(x-x_0)\bigr)
\]
has a fixed sign near $x_0$, so $h+\varepsilon=0$ has at most one nearby real
solution.  On the other hand, Rouch\'e's theorem shows that
$C+\varepsilon D$ has exactly $q$ zeros in the disk, counted with multiplicity,
for all sufficiently small $\varepsilon$.  In either parity, some of these zeros
are nonreal for a suitable real $\varepsilon$, contradicting projective
hyperbolicity of every pencil member.

Thus every nonzero member of the real pencil is simple.  In particular,
$\widetilde A$ and $\widetilde B$ have simple roots; coprimeness makes the two root
sets disjoint.  Weak interlacing is therefore strict alternating interlacing.  In any
affine chart, their Wronskian
\[
 W=\widetilde A'\widetilde B-\widetilde A\widetilde B'
\]
has no real zero.  Indeed, at a point $x$ with $\widetilde B(x)\ne0$, a zero
of $W$ would make
\[
 \widetilde A-\frac{\widetilde A(x)}{\widetilde B(x)}\widetilde B
\]
have a multiple zero at $x$; at a zero of $\widetilde B$, simplicity and
coprimeness give $W\ne0$.  Hence
$W$ has a fixed nonzero sign in each affine chart.  Applying $\phi$ proves the
claims on $\T$.
\end{proof}

\subsection{Strip contraction and the strict central difference}

For $c>0$ and $\vartheta\in\R$, define
\begin{equation}
 (\Delta_{\vartheta,c}p)(z)
 :=e^{-i\vartheta/2}p(z+ic)+e^{i\vartheta/2}p(z-ic).
 \label{eq:Delta-theta-c}
\end{equation}
If $p$ is real entire, then so is $\Delta_{\vartheta,c}p$.

The contraction estimate in the next lemma is de Bruijn's strip theorem in the
present normalization; see~\cite[Theorems~5, 6, and~8]{deBruijnTrig}.  We include
the direct argument because the proof also records the strict final-step statement
needed for the sampling obstruction.  The real-zero case of the latter is closely
related to the finite-difference Hermite--Poulain results of
Katkova--Tyaglov--Vishnyakova~\cite{KatkovaTyaglovVishnyakova}.

\begin{lemma}[De Bruijn contraction and a strict final step]
\label{lem:strict-debruijn}
Let $p$ be a real entire function of order at most one and let $c>0$.
\begin{enumerate}
\item If $\Zset(p)\subseteq\{z:|\im z|\le b\}$ and
$\Delta_{\vartheta,c}p\not\equiv0$, then
\begin{equation}
 \Zset(\Delta_{\vartheta,c}p)
 \subseteq
 \left\{z:|\im z|\le\sqrt{\max(b^2-c^2,0)}\right\}.
 \label{eq:strip-contraction}
\end{equation}
\item If $b<c$ and $\Zset(p)\subseteq\{z:|\im z|\le b\}$, then every nonzero
$\Delta_{\vartheta,c}p$ has only simple real zeros.
\end{enumerate}
\end{lemma}

\begin{proof}
If $p\equiv0$, both assertions are vacuous.  Hence assume $p\not\equiv0$.
If $p$ is zero-free, Hadamard factorization and reality give
$p(z)=Ce^{az}$ with $C,a\in\R$.  Then $\Delta_{\vartheta,c}p$ is either
identically zero or a nonzero constant multiple of $e^{az}$, so both assertions
are immediate.  Hence we may assume that $p$ has at least one zero.

We first make the polynomial approximation used in part~(1) explicit.  Define
$G(s):=p(is)$.  Since $p$ is real entire,
\[
 \J G(s)=\overline{p(-i\bar s)}=p(is)=G(s),
\]
and the zeros of $G$ lie in the vertical strip $S_b$.  Apply
Theorem~\ref{thm:balanced-approx} with the closed set $S_b$ to obtain
$\J$-invariant polynomials $q_n\to G$, all of whose zeros lie in $S_b$.  Then
\[
 p_n(z):=q_n(-iz)
\]
is a real polynomial, $p_n\to p$ locally uniformly, and
$\Zset(p_n)\subseteq\{|\,\im z|\le b\}$.  Indeed,
$q_n=\J q_n$ gives $p_n(\bar z)=\overline{p_n(z)}$.  Since $p$ has a zero,
Hurwitz's theorem shows that $p_n$ is nonconstant for all sufficiently large $n$;
discarding finitely many initial terms, we assume this for every $n$.

Fix $z=x+iy$ with
\[
 y>\sqrt{\max(b^2-c^2,0)}.
\]
For a real zero $u$ of $p_n$,
\[
 |z+ic-u|^2-|z-ic-u|^2=4cy>0.
\]
For a conjugate pair $u\pm iv$, where $0<v\le b$, direct expansion gives
\begin{align*}
 &\prod_{\pm}|z+ic-(u\pm iv)|^2
 -\prod_{\pm}|z-ic-(u\pm iv)|^2\\
 &\hspace{28mm}=8cy\bigl((x-u)^2+y^2+c^2-v^2\bigr)>0.
\end{align*}
The last inequality follows from $v^2\le b^2<y^2+c^2$.  Multiplying over all
real zeros and conjugate pairs yields
\[
 |p_n(z+ic)|>|p_n(z-ic)|.
\]
Because the two coefficients in $\Delta_{\vartheta,c}$ have the same modulus,
$\Delta_{\vartheta,c}p_n(z)\ne0$ in this upper region.  The reverse modulus
inequality holds in the reflected lower region.  Since
$\Delta_{\vartheta,c}p_n\to\Delta_{\vartheta,c}p$ locally uniformly, Hurwitz's
theorem proves~\eqref{eq:strip-contraction}; the excluded alternative that the
limit vanish on an open disk would imply
$\Delta_{\vartheta,c}p\equiv0$.

Now assume $b<c$.  Part~(1) shows that every zero of a nonzero
$\Delta_{\vartheta,c}p$ is real.  The lines $\im z=\pm c$ are disjoint from
$\Zset(p)$, so $p(x+ic)\ne0$ for every $x\in\R$.  Group a genus-one canonical
product for $p$ into real zeros $u$
and conjugate pairs $u\pm iv$, $0<v\le b$.  Reality of $p$ makes the residual
exponential factor $Ce^{az}$ with $C,a\in\R$, and the reciprocal terms in each
conjugate pair have real sum.  Taking the imaginary part of the logarithmic
derivative at $x+ic$ therefore gives
\begin{align}
 \frac d{dx}\arg p(x+ic)
 &=\im\frac{p'(x+ic)}{p(x+ic)}\notag\\
 &=-\sum_{u\in\R}\frac{c}{(x-u)^2+c^2}\notag\\
 &\quad-\sum_{u+iv,\ v>0}
 \left(
 \frac{c-v}{(x-u)^2+(c-v)^2}
 +\frac{c+v}{(x-u)^2+(c+v)^2}
 \right)<0.
 \label{eq:strict-phase-derivative}
\end{align}
Zeros and multiplicities are included in the sums.  The genus-one logarithmic
derivative converges locally uniformly on the line $\im z=c$, which stays a
positive distance from the zero strip, so termwise evaluation is justified.  Every
summand is nonpositive and at least one is strictly negative.

Set
\[
 Z(x):=e^{-i\vartheta/2}p(x+ic).
\]
Since $p$ is real entire,
\[
 \Delta_{\vartheta,c}p(x)=Z(x)+\overline{Z(x)}=2\re Z(x).
\]
At a zero of this real-valued function, $Z(x)\ne0$ and its argument is an odd
multiple of $\pi/2$.  Writing $Z=|Z|e^{i\phi}$, equation
\eqref{eq:strict-phase-derivative} gives $\phi'(x)<0$, and hence
\[
 \frac d{dx}\re Z(x)=-|Z(x)|\phi'(x)\sin\phi(x)\ne0
\]
at the zero.  Thus every zero is simple.
\end{proof}

\subsection{The general strict sampling theorem}

For $\zeta\in\T$ and $\delta>0$, put
\begin{equation}
 \mathscr T_{\zeta,\delta}:=e^{-\delta D/2}+\zeta e^{\delta D/2},
 \qquad D=\frac d{ds}.
 \label{eq:T-zeta-delta}
\end{equation}
The binomial theorem gives
\begin{equation}
 \mathscr T_{\zeta,\delta}^{\,d}E(0)=B_{d,\delta}[E](\zeta),
 \qquad
 \bigl(\mathscr T_{\zeta,\delta}^{\,d}E\bigr)'(0)
 =B_{d,\delta}[E'](\zeta).
 \label{eq:T-sampling-identities}
\end{equation}
The identities show that a common sampled zero is the same as a double zero at the
origin of one finite-difference iterate.  The classical contraction theorem does not
by itself rule out such a double zero, because an iterate may vanish identically.
The next proposition is the additional sampling step: it proves that this identically
zero case is the only obstruction.  This is the point at which the de Bruijn input is
converted into coprimeness of two consecutive sampled derivatives.

\begin{proposition}[Exact obstruction to strictness]
\label{prop:exact-strict-obstruction}
Let $d\ge2$, $\delta>0$, and $0\le h<\delta\sqrt d/2$.  Let $E\not\equiv0$ be
balanced, entire of order at most one, and satisfy $\Zset(E)\subseteq S_h$.  Put
\[
 A:=B_{d,\delta}[E],
 \qquad C:=B_{d,\delta}[E'].
\]
Then, for every $\zeta\in\T$,
\begin{equation}
 A(\zeta)=C(\zeta)=0
 \quad\Longleftrightarrow\quad
 \mathscr T_{\zeta,\delta}^{\,d}E\equiv0.
 \label{eq:exact-strict-obstruction}
\end{equation}
Consequently, the nondegeneracy condition
$\mathscr T_{\zeta,\delta}^{\,d}E\not\equiv0$ for all $\zeta\in\T$ is equivalent to
coprimeness of $A$ and $C$.
\end{proposition}

\begin{proof}
The reverse implication in~\eqref{eq:exact-strict-obstruction} follows immediately
from~\eqref{eq:T-sampling-identities}.  For the converse, suppose that
$A(\zeta)=C(\zeta)=0$, write $\zeta=e^{i\vartheta}$, and set
\[
 H:=\mathscr T_{\zeta,\delta}^{\,d}E.
\]
Then $H(0)=H'(0)=0$.  We show that a nonzero $H$ cannot have a multiple zero.

After replacing $E$ and $H$ by the same nonzero constant multiple, we may assume
that $E$ is $\J$-invariant.  Put $p(z)=E(iz)$.  Then $p$ is real entire of order at
most one and $\Zset(p)\subseteq\{|\im z|\le h\}$.  With $c=\delta/2$, direct
substitution gives
\[
 \bigl(\mathscr T_{\zeta,\delta}E\bigr)(iz)
 =p(z+ic)+e^{i\vartheta}p(z-ic)
 =e^{i\vartheta/2}\Delta_{\vartheta,c}p(z).
\]
The two translation operators commute, so iteration yields the exact identity
\begin{equation}
 \bigl(\mathscr T_{\zeta,\delta}^{\,r}E\bigr)(iz)
 =e^{ir\vartheta/2}\Delta_{\vartheta,c}^{\,r}p(z)
 \qquad(r\ge0).
 \label{eq:T-Delta-intertwining}
\end{equation}
If $H$ is nonzero, every preceding central-difference iterate is nonzero.  Repeated
application of Lemma~\ref{lem:strict-debruijn}(1) shows that after $d-1$ iterations
the zero strip has half-width at most
\[
 b_{d-1}:=\sqrt{\max(h^2-(d-1)c^2,0)}<c,
\]
because $h<\sqrt d\,c$.  Hence the final nonzero difference has only simple real
zeros by Lemma~\ref{lem:strict-debruijn}(2).  Equation~\eqref{eq:T-Delta-intertwining}
then implies that every zero of $H$ is simple, contradicting $H(0)=H'(0)=0$.
Therefore $H\equiv0$, proving~\eqref{eq:exact-strict-obstruction}.

It remains only to translate the pointwise obstruction into coprimeness.  If
$E'\not\equiv0$, then Corollary~\ref{cor:differentiation} gives
$\Zset(E')\subseteq S_h$, and Corollary~\ref{cor:scaled}, applied to $E$ and to
$E'$, places both $A$ and $C$ in $\mathcal C_d$.  Thus every zero of either
polynomial lies on $\T$, and the two polynomials have a common factor if and only
if they have a common zero on $\T$.  The equivalence
\eqref{eq:exact-strict-obstruction} rules out exactly those common zeros.  If
$E'\equiv0$, then $E$ is a nonzero constant, $C\equiv0$, and
$\mathscr T_{-1,\delta}^{\,d}E\equiv0$; in this case the nondegeneracy
condition fails, and $A$ and $C$ are not coprime under the convention that the zero
polynomial is not coprime to any polynomial.  This proves the final assertion.
\end{proof}

\begin{theorem}[Strict strip-to-circle sampling]
\label{thm:strict-sampling}
Let $d\ge2$, $\delta>0$, and
\[
 0\le h<\frac{\delta\sqrt d}2.
\]
Let $F\not\equiv0$ be an entire function of order at most one such that
\[
 \Zset(F)\subseteq S_h,
 \qquad F=\omega\J F,
 \qquad |\omega|=1.
\]
Fix $m\ge0$ such that $F^{(m)}\not\equiv0$.  The following two conditions are
equivalent:
\begin{enumerate}
\item[(a)]
\begin{equation}
 \mathscr T_{\zeta,\delta}^{\,d}F^{(m)}\not\equiv0
 \qquad\text{for every }\zeta\in\T;
 \label{eq:strict-nondegeneracy}
\end{equation}
\item[(b)] the polynomials $B_{d,\delta}[F^{(m)}]$ and
$B_{d,\delta}[F^{(m+1)}]$ are coprime.
\end{enumerate}
Whenever these equivalent conditions hold, the following conclusions hold.
\begin{enumerate}
\item The polynomials
\[
 A_m:=B_{d,\delta}[F^{(m)}],
 \qquad A_{m+1}:=B_{d,\delta}[F^{(m+1)}]
\]
have only simple zeros on $\T$ and their zeros strictly cyclically interlace.
\item For every $t\in\R$, $A_m+itA_{m+1}$ has $d$ simple zeros on $\T$.
\item After a fixed Cayley transformation, the roots of this pencil form monotone
projective root branches; in any affine chart in which a branch is finite, all root
derivatives have the same sign.
\end{enumerate}
\end{theorem}

\begin{proof}
Put $E=F^{(m)}$, which is nonzero by hypothesis.  It is balanced with phase
$(-1)^m\omega$ and has all zeros in $S_h$ by
Corollary~\ref{cor:differentiation}.  Proposition~\ref{prop:exact-strict-obstruction}
proves the equivalence of (a) and (b).  Assume these conditions.  The polynomial
$A_m$ and the polynomial $iA_{m+1}$ have the same self-inversive phase.

We first show that every nonzero real linear combination of $A_m$ and $iA_{m+1}$
belongs to $\mathcal C_d$.  Let $a,b\in\R$ be not both zero.  If $b=0$, the
corresponding function $aE$ is nonzero and has all its zeros in $S_h$.  If $a=0$,
condition~\eqref{eq:strict-nondegeneracy} rules out a nonzero constant $E$ (take
$\zeta=-1$), so $E'\not\equiv0$; Corollary~\ref{cor:differentiation} then puts all
zeros of $ibE'$ in $S_h$.  Finally, suppose $ab\ne0$.  An identity
$aE+ibE'\equiv0$ would give $E'=i(a/b)E$ and hence
$E(s)=Ce^{i(a/b)s}$.  Choosing $\zeta=-e^{-i(a/b)\delta}$ would then make
$\mathscr T_{\zeta,\delta}E\equiv0$, contrary to
\eqref{eq:strict-nondegeneracy}.  Thus the function is nonzero, and after division
by $a$ Lemma~\ref{lem:derivative-pencil-strip}, with $t=b/a$, places all its zeros
in $S_h$.

In all three cases the corresponding function is balanced with the common phase of
$E$ and $iE'$.  Since $h<\delta\sqrt d/2$, Corollary~\ref{cor:scaled} gives
\[
 aA_m+biA_{m+1}\in\mathcal C_d.
\]
The circular Obreschkoff lemma therefore yields weak cyclic interlacing.
Proposition~\ref{prop:exact-strict-obstruction} and
\eqref{eq:strict-nondegeneracy} show that $A_m$ and $iA_{m+1}$ are coprime.
Lemma~\ref{lem:circular-obreschkoff} proves
parts~(1) and~(2).  For part~(3), apply the Cayley transformation from the proof of
that lemma and dehomogenize in an affine chart.  Let $\widetilde A,\widetilde B$ be
the resulting real polynomials.  Strict interlacing implies that
\[
 W(x):=\widetilde A'(x)\widetilde B(x)
       -\widetilde A(x)\widetilde B'(x)
\]
has a fixed nonzero sign.  If $x_j(t)$ is a finite root of
$\widetilde A+t\widetilde B$, implicit differentiation and the root equation give
\begin{equation}
 x_j'(t)
 =-\frac{\widetilde B(x_j(t))^2}{W(x_j(t))}.
 \label{eq:root-flow-derivative}
\end{equation}
All finite branches therefore move in the same direction.  Projective continuation
through infinity gives the monotone cyclic flow on $\T$.
\end{proof}

\begin{example}[Sharpness of the open strict threshold]
\label{ex:strict-endpoint-sharp}
The strict inequality in Theorem~\ref{thm:strict-sampling} cannot be replaced, uniformly
in $d$, by a weak one.  Take $d=2$ and
\[
 F(s)=s^2-\frac{\delta^2}{2}.
\]
Then $F$ is balanced, its zeros lie at the two boundary points
$\pm\delta/\sqrt2$ of the strip with half-width $\delta\sqrt2/2$, and
\[
 B_{2,\delta}[F](z)=\frac{\delta^2}{2}(z-1)^2.
\]
Thus simplicity fails at the endpoint.  The nondegeneracy hypothesis is nevertheless
satisfied: for $\zeta\ne-1$, the coefficient of $s^2$ in
$\mathscr T_{\zeta,\delta}^{\,2}F$ is $(1+\zeta)^2$, while
\[
 \mathscr T_{-1,\delta}^{\,2}F(s)
 =F(s-\delta)-2F(s)+F(s+\delta)=2\delta^2.
\]
Hence the open strip threshold in the theorem is sharp as a degree-uniform statement.
\end{example}

\begin{remark}[The obstruction is exact]
\label{rem:nondegeneracy-use}
Proposition~\ref{prop:exact-strict-obstruction} shows that
\eqref{eq:strict-nondegeneracy} is not merely a technical sufficient hypothesis: it
is exactly the condition excluding common zeros of two consecutive sampled
derivatives.  Hence, under the remaining hypotheses of
Theorem~\ref{thm:strict-sampling}, it is also necessary and sufficient for strict
cyclic interlacing.  In the completed $L$-function applications below, a right-edge
asymptotic verifies this nondegeneracy directly.
\end{remark}

\section{Completed functional equations and derivative period polynomials}

We now transfer the general sampling theory to completed $L$-functions.  Centering a functional equation at its symmetry point turns it into the balance relation used in the previous sections, and a zero strip becomes a centered strip for the sampling function.  We first record this general principle, then apply it to Riemann's $\xi$-function and to primitive holomorphic newforms.

\begin{theorem}[Sampling from a completed functional equation]
\label{thm:functional-equation-sampling}
Let $w\in\R$, $h\ge0$, $\delta>0$, and $d\ge2$.  Let $\Xi\not\equiv0$ be entire of
order at most one and suppose
\begin{equation}
 \Xi(s)=\epsilon\,\overline{\Xi(w-\bar s)},
 \qquad |\epsilon|=1,
 \label{eq:general-functional-equation}
\end{equation}
and
\begin{equation}
 \Zset(\Xi)\subseteq
 \left\{s:\left|\re s-\frac w2\right|\le h\right\}.
 \label{eq:general-zero-strip}
\end{equation}
If $h\le\delta\sqrt d/2$, then for every $m\ge0$ with
$\Xi^{(m)}\not\equiv0$,
\begin{equation}
 P_{\Xi,m}^{(d,\delta)}(z)
 :=\sum_{j=0}^d\binom dj
 \Xi^{(m)}\!\left(\frac w2+\delta\left(j-\frac d2\right)\right)z^j
 \label{eq:general-L-sampling}
\end{equation}
belongs to $\mathcal C_d$.  Its binomial coefficient vector is a circular multiplier
sequence of degree $d$.  If $\Xi$ is not a polynomial, the conclusion holds for all
$m\ge0$.
\end{theorem}

\begin{proof}
Set $F(u)=\Xi(w/2+u)$.  Then $F=\epsilon\J F$ and
$\Zset(F)\subseteq S_h$, while~\eqref{eq:general-L-sampling} is
$B_{d,\delta}[F^{(m)}]$.  Corollary~\ref{cor:scaled} proves the result.  If a
derivative of $\Xi$ were identically zero, then $\Xi$ would be a polynomial.
\end{proof}

\subsection{Riemann's \texorpdfstring{$\xi$}{xi}-function}

Let
\[
 \xi(s)=\frac12s(s-1)\pi^{-s/2}\Gamma(s/2)\zeta(s),
 \qquad
 X_{d,m}(z)=\sum_{j=0}^d\binom dj
 \xi^{(m)}\!\left(j-\frac{d-1}{2}\right)z^j.
\]

\begin{corollary}[Strict Riemann $\xi$-sampling]
\label{cor:strict-xi-sampling}
For every $d\ge2$ and $m\ge0$, $X_{d,m}$ has $d$ simple zeros on $\T$.  The zeros of $X_{d,m}$ and $X_{d,m+1}$ strictly cyclically interlace, and every pencil
\[
 X_{d,m}+itX_{d,m+1},\qquad t\in\R,
\]
has simple unit-circle zeros forming a monotone cyclic root flow.  No hypothesis on the zeros of $\zeta$ is used.
\end{corollary}

\begin{proof}
The function $\xi$ is a nonpolynomial entire function of order one, satisfies $\xi(s)=\xi(1-s)$, and has all zeros in $0\le\re s\le1$.  The location assertion follows from Theorem~\ref{thm:functional-equation-sampling} with $w=1$, $h=1/2$, and $\delta=1$.

For strictness, put $F_m(u)=\xi^{(m)}(1/2+u)$.  Stirling's formula gives, for fixed real $a,b$,
\[
 \frac{F_m(x+a)}{F_m(x+b)}
 =\left(\frac{x}{2\pi}\right)^{(a-b)/2}(1+o(1))
 \qquad(x\to+\infty).
\]
Expanding the commuting translations gives
\[
 \mathscr T_{\zeta,1}^{\,r}F_m(x)
 =\sum_{j=0}^{r}\binom rj\zeta^j
 F_m\!\left(x+j-\frac r2\right).
\]
The term $j=r$ has the largest shift.  Dividing by
$\zeta^rF_m(x+r/2)$ and using the displayed ratio for the finitely many shifts
$j-r$ gives, uniformly for $\zeta\in\T$,
\[
 \frac{\mathscr T_{\zeta,1}^{\,r}F_m(x)}
 {\zeta^rF_m(x+r/2)}
 =1+O_{m,r}(x^{-1/2}).
\]
Hence the iterate is nonzero for every sufficiently large $x$ and cannot vanish
identically.  In particular, the case $r=0$ shows $F_m\not\equiv0$, as required
in Theorem~\ref{thm:strict-sampling}.  That theorem now gives all strict
conclusions.
\end{proof}

\subsection{Primitive newforms and derivative period polynomials}
\label{sec:newforms}

Let
\[
 f(\tau)=\sum_{n\ge1}a_f(n)e^{2\pi in\tau}
 \in S_k^{\mathrm{new}}(\Gamma_0(N),\chi)
\]
be a normalized primitive holomorphic newform of weight $k\ge4$, level $N$, and
nebentypus $\chi$.  Write $\bar f$ for the newform obtained by conjugating the
Fourier coefficients, and define
\begin{equation}
 L(f,s):=\sum_{n\ge1}\frac{a_f(n)}{n^s},
 \qquad
 \Lambda(f,s):=\left(\frac{\sqrt N}{2\pi}\right)^s\Gamma(s)L(f,s).
 \label{eq:completed-L}
\end{equation}
The completed function is entire of order one and satisfies
\begin{equation}
 \Lambda(f,s)=\eps_f\Lambda(\bar f,k-s)
 =\eps_f\overline{\Lambda(f,k-\bar s)},
 \qquad |\eps_f|=1.
 \label{eq:functional-equation}
\end{equation}
These normalizations agree with Liu--Park--Song \cite{LiuParkSong}.  We use the
standard analytic theory of primitive holomorphic newform $L$-functions; see, for
example, Miyake~\cite{Miyake} and Iwaniec--Kowalski~\cite{IwaniecKowalski}.
Deligne's Ramanujan--Petersson bound~\cite{Deligne} makes the Euler product
absolutely convergent and nonvanishing for
$\re s>(k+1)/2$; the functional equation then gives
\begin{equation}
 \Zset(\Lambda(f,\cdot))
 \subseteq\left\{s:\left|\re s-\frac k2\right|\le\frac12\right\}.
 \label{eq:L-zero-strip}
\end{equation}
Set
\begin{equation}
 F_f(s):=\Lambda\!\left(f,\frac k2+s\right),
 \qquad d:=k-2.
 \label{eq:Ff}
\end{equation}
Then $F_f=\eps_f\J F_f$ and $\Zset(F_f)\subseteq S_{1/2}$.

\begin{theorem}[Circular critical-value multipliers]
\label{thm:newform-main}
For every $m\ge0$, define
\begin{equation}
 U_{f,m}(z):=\sum_{j=0}^{k-2}\binom{k-2}{j}\Lambda^{(m)}(f,j+1)z^j.
 \label{eq:Ufm}
\end{equation}
Then $U_{f,m}\in\mathcal C_{k-2}$.  More strongly, its critical-value vector is a
degree-$(k-2)$ circular multiplier sequence.
\end{theorem}

\begin{proof}
For $0\le j\le d$,
\[
 F_f^{(m)}\!\left(j-\frac d2\right)=\Lambda^{(m)}(f,j+1).
\]
No derivative is identically zero, since otherwise $\Lambda(f,\cdot)$ would be a
polynomial, contradicting Stirling growth and $L(f,\sigma)\to1$.  Apply
Theorem~\ref{thm:functional-equation-sampling}.
\end{proof}

For strictness, write
\[
 F_m(s):=\Lambda^{(m)}\!\left(f,\frac k2+s\right),
 \qquad
 \mathscr T_\zeta:=\mathscr T_{\zeta,1}.
\]

\begin{lemma}[Gamma-growth nondegeneracy for newforms]
\label{lem:gamma-growth-nondegeneracy}
For every $m,r\ge0$ and every $\zeta\in\T$,
\[
 \mathscr T_\zeta^{\,r}F_m\not\equiv0.
\]
\end{lemma}

\begin{proof}
Put $A=\sqrt N/(2\pi)$ and
\[
 G(s):=A^s\Gamma(s),
 \qquad \Lambda(f,s)=G(s)L(f,s).
\]
We first record the asymptotic with enough uniformity for the finitely many shifts
that occur below.  For each fixed integer $q\ge0$, absolute convergence of the
Dirichlet series in a right half-plane gives
\begin{equation}
 L(f,\sigma)=1+o(1),
 \qquad L^{(q)}(f,\sigma)=o(1)\quad(q\ge1)
 \qquad(\sigma\to+\infty).
 \label{eq:newform-L-right-edge}
\end{equation}
The error is uniform when $\sigma$ is replaced by $\sigma+a$ with $a$ in any fixed
finite set.  On the other hand, the complete Bell-polynomial formula for derivatives
of $G$ and the standard polygamma estimates imply, for every fixed $\ell\ge0$,
\begin{equation}
 \frac{G^{(\ell)}(\sigma)}{G(\sigma)}
 =\bigl(\log(A\sigma)\bigr)^\ell(1+o(1)).
 \label{eq:newform-gamma-derivative-edge}
\end{equation}
Indeed,
\[
 \psi(\sigma)=\log\sigma+O(\sigma^{-1}),
 \qquad
 \psi^{(j)}(\sigma)=O_j(\sigma^{-j})\quad(j\ge1).
\]
The leading Bell monomial is $(\log A+\psi(\sigma))^\ell$; every other
monomial contains at least one derivative $\psi^{(j)}$ with $j\ge1$ and is
$o((1+\log\sigma)^\ell)$.  Again, the estimate is uniform over a fixed finite
set of additive shifts.

Leibniz' rule, \eqref{eq:newform-L-right-edge}, and
\eqref{eq:newform-gamma-derivative-edge} now give, for fixed real $a$,
\begin{equation}
 F_m(x+a)
 =G(x+a+k/2)\bigl(\log(Ax)\bigr)^m(1+o(1))
 \qquad(x\to+\infty).
 \label{eq:newform-right-growth}
\end{equation}
For $m=0$ the logarithmic factor is interpreted as $1$.  Combining this with the
fixed-shift gamma-ratio formula yields, uniformly for $a,b$ in any prescribed finite
set,
\begin{equation}
 \frac{F_m(x+a)}{F_m(x+b)}
 =(Ax)^{a-b}(1+o(1)).
 \label{eq:newform-shift-ratio}
\end{equation}
In particular, all denominators in this formula are nonzero for sufficiently large
$x$.

For $r\ge0$, the commuting translations give the exact expansion
\begin{equation}
 \mathscr T_\zeta^{\,r}F_m(x)
 =\sum_{j=0}^{r}\binom rj\zeta^j
 F_m\!\left(x+j-\frac r2\right).
 \label{eq:T-zeta-binomial-expansion}
\end{equation}
The term $j=r$ has the largest shift.  Dividing by that term and applying
\eqref{eq:newform-shift-ratio} to the finite set
$\{-r/2,-r/2+1,\ldots,r/2\}$ gives
\begin{align}
 \frac{\mathscr T_\zeta^{\,r}F_m(x)}
 {\zeta^rF_m(x+r/2)}
 &=1+\sum_{j=0}^{r-1}\binom rj\zeta^{j-r}
 \frac{F_m(x+j-r/2)}{F_m(x+r/2)}\notag\\
 &=1+O_{f,m,r}(x^{-1}).
 \label{eq:T-zeta-dominant-shift}
\end{align}
The last error is uniform for $\zeta\in\T$, since $|\zeta|=1$; for $r=0$ the sum is
empty.  Thus the quotient in~\eqref{eq:T-zeta-dominant-shift} tends to $1$, so the
iterate is nonzero for every sufficiently large real $x$.  It therefore cannot be
the zero entire function.
\end{proof}

\begin{theorem}[Global simplicity, interlacing, and pencil flow]
\label{thm:global-newform-interlacing}
For every primitive holomorphic newform of weight $k\ge4$, arbitrary level and
nebentypus, and every $m\ge0$, the following hold.
\begin{enumerate}
\item $U_{f,m}$ has $k-2$ simple zeros on $\T$.
\item The zeros of $U_{f,m}$ and $U_{f,m+1}$ strictly cyclically interlace.
\item Every $U_{f,m}+itU_{f,m+1}$, $t\in\R$, has simple zeros on $\T$, and these
zeros form a monotone cyclic root flow.
\end{enumerate}
\end{theorem}

\begin{proof}
Here $h=1/2$, $\delta=1$, and $d=k-2\ge2$, so
$h<\sqrt d/2$.  Lemma~\ref{lem:gamma-growth-nondegeneracy}, applied also with
$r=0$, shows that $F_m\not\equiv0$ and verifies the nondegeneracy hypothesis of
Theorem~\ref{thm:strict-sampling}, which gives all three claims.
\end{proof}

\begin{corollary}[Endpoint and two-order inequalities]
\label{cor:newform-endpoint-turan}
Put
\[
 a_j=\Lambda^{(m)}(f,j+1),
 \qquad b_j=\Lambda^{(m+1)}(f,j+1).
\]
For $1\le j\le k-3$ and every $t\in\R$,
\begin{equation}
 |a_j+itb_j|<|a_0+itb_0|.
 \label{eq:pencil-strict-domination}
\end{equation}
In particular,
\begin{equation}
 |\Lambda^{(m)}(f,j+1)|
 <|\Lambda^{(m)}(f,1)|
 =|\Lambda^{(m)}(f,k-1)|,
 \label{eq:newform-strict-endpoint-inequality}
\end{equation}
and
\begin{align}
 &\left(\im(a_0\overline{b_0})-\im(a_j\overline{b_j})\right)^2\notag\\
 &\qquad<
 \left(|a_0|^2-|a_j|^2\right)
 \left(|b_0|^2-|b_j|^2\right).
 \label{eq:Turan-two-order}
\end{align}
\end{corollary}

\begin{proof}
If $P(z)=\sum_{j=0}^d\binom djc_jz^j$ has simple zeros on $\T$, then
$|c_j|<|c_0|=|c_d|$ for $1\le j\le d-1$: coefficient comparison expresses
$\binom djc_j$ as an elementary symmetric sum of the unimodular roots, and equality
in the triangle inequality would force all roots to coincide.  Apply this observation
to every pencil $U_{f,m}+itU_{f,m+1}$, which has simple unit-circle zeros by
Theorem~\ref{thm:global-newform-interlacing}.  This gives
\eqref{eq:pencil-strict-domination} and then~\eqref{eq:newform-strict-endpoint-inequality}.
Applying the same coefficient comparison to $U_{f,m+1}$ gives
$|b_j|<|b_0|$ for $1\le j\le d-1$.  Thus, after squaring
\eqref{eq:pencil-strict-domination}, the difference
\[
 |a_0+itb_0|^2-|a_j+itb_j|^2
\]
is a real quadratic in $t$ with positive leading coefficient
$|b_0|^2-|b_j|^2$.  Since it is positive for every real $t$, its
discriminant is negative, which is exactly~\eqref{eq:Turan-two-order}.
\end{proof}

Define the reciprocal-rotation normalization
\begin{equation}
 Q_{f,m}^{\mathrm{norm}}(z)
 :=\sum_{j=0}^{k-2}\binom{k-2}{j}i^{1-j}
 \Lambda^{(m)}(f,j+1)z^{k-2-j}.
 \label{eq:Qnorm}
\end{equation}
Then
\begin{equation}
 Q_{f,m}^{\mathrm{norm}}(z)
 =iz^{k-2}U_{f,m}\!\left(\frac1{iz}\right).
 \label{eq:reciprocal-rotation}
\end{equation}
Put
\begin{equation}
 R_{f,m}^{(N)}(X)
 :=(-1)^{k-2}N^{-(k-1)/2}Q_{f,m}^{\mathrm{norm}}(\sqrt N X).
 \label{eq:Rfm}
\end{equation}

\begin{corollary}[Conductor-rescaled derivative period polynomials]
\label{cor:conductor-rescaled}
Every $Q_{f,m}^{\mathrm{norm}}$ has simple zeros on $\T$, and consecutive derivative
orders strictly cyclically interlace.  Every $R_{f,m}^{(N)}$ has simple zeros on
$|X|=N^{-1/2}$, with the same interlacing and pencil-flow conclusions.

For $m=0$, $R_{f,0}^{(N)}$ is the classical period polynomial
\[
 r_f(X)=\int_0^{i\infty}f(\tau)(\tau-X)^{k-2}\,d\tau.
\]
\end{corollary}

\begin{proof}
Reciprocal rotation and radial rescaling preserve multiplicities and cyclic interlacing, up to reversal of orientation.
For $m=0$, the standard Mellin calculation gives
\begin{equation}
 r_f(X)
 =i^{k-1}N^{-(k-1)/2}
 \sum_{j=0}^{k-2}\binom{k-2}{j}(\sqrt N\,iX)^j\Lambda(f,k-1-j),
 \label{eq:period-Mellin}
\end{equation}
see~\cite[(1.3)]{LiuParkSong}.  Reindexing~\eqref{eq:Qnorm} gives exactly
\eqref{eq:Rfm} after using $i^{2(k-2)}=(-1)^{k-2}$.
\end{proof}

\begin{corollary}[Diamantis--Rolen full-polynomial assertion]
\label{cor:DR-full}
Let $f$ be a normalized Hecke eigen cusp form on $\mathrm{SL}_2(\mathbb Z)$ of even
weight $k$.  For every $m\ge0$, the full derivative period polynomial
\[
 Q_{f,m}(z)=\sum_{j=0}^{k-2}\binom{k-2}{j}i^{1-j}
 \Lambda^{(m)}(f,j+1)z^{k-2-j}
\]
has simple unit-circle zeros.  The zeros for orders $m$ and $m+1$ strictly cyclically
interlace, and the corresponding pencil has a monotone cyclic root flow.
\end{corollary}

\begin{remark}
The full-polynomial assertion is thereby strengthened.  The odd-part assertion is a
separate problem and does not follow formally from circular sampling, since odd
projection is not a circular multiplier.  No odd-part input is used in this paper.
\end{remark}

\begin{remark}
Theorem~\ref{thm:global-newform-interlacing} also implies, by the strict boundary case of Gauss--Lucas, that every zero of $U_{f,m}'$ lies in the open unit disk.
\end{remark}

\section{Conductor-uniform quantitative localization and spacing}
\label{sec:quantitative-localization}

The preceding section proves simplicity and consecutive-order interlacing for
arbitrary level, nebentypus, weight, and derivative order, but gives no quantitative
control of the angular gaps.  We now prove Theorem~\ref{thm:intro-quantitative}.
The argument treats even and odd weights simultaneously and is uniform in the
conductor.  For derivative order zero the phase model is the classical one appearing
in~\cite{JinMaOnoSound}; in the arbitrary-nebentypus setting, see also~\cite{LiuParkSong},
where finitely many possible exceptional newforms are allowed.  Our purpose is to propagate this model through every
fixed derivative order and to compare two consecutive orders quantitatively.

Let
\[
 f\in S_k^{\mathrm{new}}(\Gamma_0(N),\chi)
\]
be a normalized primitive holomorphic newform, let $\mu\ge0$ be fixed, and write
\[
 d=k-2,\qquad D=\frac d2,\qquad J=\left\lfloor\frac d2\right\rfloor,
 \qquad a_N=\frac{2\pi}{\sqrt N}.
\]
Set
\[
 G_N(s):=a_N^{-s}\Gamma(s)
 =\left(\frac{\sqrt N}{2\pi}\right)^s\Gamma(s),
 \qquad
 A_{f,\mu}(s):=\frac{\Lambda^{(\mu)}(f,s)}{G_N(s)}.
\]
Choose $\omega_f\in\R$ with $\eps_f=e^{i\omega_f}$ and put
\[
 \eta_{f,\mu}:=(-1)^\mu\eps_f,
 \qquad
 \Gamma_{f,\mu}:=\frac{\omega_f+\mu\pi}{2}.
\]
Differentiating the functional equation gives
\begin{equation}
 \Lambda^{(\mu)}(f,s)
 =\eta_{f,\mu}\,
  \overline{\Lambda^{(\mu)}(f,k-\bar s)}.
 \label{eq:quantitative-derivative-FE}
\end{equation}

Define
\begin{equation}
 c_r:=A_{f,\mu}(k-1-r)\frac{a_N^r}{r!},
 \qquad 0\le r\le d,
 \label{eq:general-cr-definition}
\end{equation}
and
\begin{equation}
 \mathcal R_{f,\mu,k,N}(w):=\sum_{r=0}^{d}c_rw^r.
 \label{eq:general-R-definition}
\end{equation}
Reindexing~\eqref{eq:Qnorm} gives
\begin{equation}
 Q_{f,\mu}^{\mathrm{norm}}(z)
 =i^{1-d}d!\,a_N^{-(k-1)}
  \mathcal R_{f,\mu,k,N}(iz).
 \label{eq:general-Q-R-rotation}
\end{equation}
The functional equation also gives an exact conjugate reciprocity.  Indeed,
\[
 \frac{G_N(k-1-r)}{G_N(r+1)}
 =a_N^{-d+2r}\frac{(d-r)!}{r!},
\]
and substitution in~\eqref{eq:general-cr-definition} yields
\begin{equation}
 c_{d-r}=\eta_{f,\mu}\,\overline{c_r}.
 \label{eq:general-c-reciprocity}
\end{equation}

Define the right-edge polynomial by
\begin{equation}
 q_{f,\mu,k,N}(w)
 :=\sum_{0\le r<d/2}c_rw^r
 +\begin{cases}
   \frac12c_{d/2}w^{d/2},& d\ \text{even},\\
   0,&d\ \text{odd}.
  \end{cases}
 \label{eq:general-q-edge}
\end{equation}
Thus, when $d$ is odd, the sum in~\eqref{eq:general-q-edge} runs through
$r=0,\ldots,J$.  In both parities,~\eqref{eq:general-c-reciprocity} gives
\begin{equation}
 \mathcal R_{f,\mu,k,N}(w)
 =q_{f,\mu,k,N}(w)
 +\eta_{f,\mu}w^d
  \overline{q_{f,\mu,k,N}(1/\bar w)}.
 \label{eq:general-edge-decomposition}
\end{equation}

We first record the deliberately crude central estimate needed only to dispose of the
last coefficient in~\eqref{eq:general-q-edge}.

\begin{lemma}[Polynomial bound in the central strip]
\label{lem:general-central-polynomial-bound}
For every fixed $\mu\ge0$ there is a constant $C_\mu>0$ such that, uniformly over
all primitive $f$ as above and all real $s$ with
\[
 \frac k2\le s\le\frac{k+1}{2},
\]
one has
\begin{equation}
 |A_{f,\mu}(s)|\ll_\mu (Nk)^{C_\mu}.
 \label{eq:general-central-polynomial-bound}
\end{equation}
\end{lemma}

\begin{proof}
Put
\[
 \mathcal L_f(u):=L\!\left(f,u+\frac{k-1}{2}\right).
\]
This is the unitary normalization of the holomorphic newform $L$-function.  Fix
$0<\rho<1/4$.  On the closed $\rho$-neighborhood of $[1/2,1]$, the analytic
conductor of $\mathcal L_f$, including the finite and archimedean factors, is
$\ll_\rho N(k+1)^2$, uniformly in the level, nebentypus, and primitive newform.
The standard GL(2) convexity bound, obtained
from the functional equation and the Phragm\'en--Lindel\"of principle, gives the
following polynomial conductor estimate;
equivalently, it follows from the approximate functional equation together with
Deligne's bound.  Thus, for some $C_\rho>0$,
\[
 \sup_{\operatorname{dist}(u,[1/2,1])\le\rho}
 |\mathcal L_f(u)|
 \ll_\rho \bigl(N(k+1)^2\bigr)^{C_\rho},
\]
uniformly in the level, nebentypus, and newform.  Cauchy's formula on circles of
radius $\rho/2$ then gives, for every fixed $b\ge0$,
\begin{equation}
 \mathcal L_f^{(b)}(u)\ll_{b,\rho} (Nk)^{C_b}
 \qquad(1/2\le u\le1)
 \label{eq:unitary-polynomial-convexity}
\end{equation}
for some $C_b$ independent of $N$, $\chi$, and $f$; see, for example, the
convexity principle and approximate functional equation in
\cite[Chapter~5]{IwaniecKowalski}.  On the same interval, the Bell-polynomial
formula and the elementary polygamma estimates give
\[
 \frac{G_N^{(j)}(s)}{G_N(s)}
 \ll_j(1+\log(Nk))^j.
\]
Leibniz' rule for $\Lambda^{(\mu)}=(G_NL)^{(\mu)}$, together with
\eqref{eq:unitary-polynomial-convexity}, proves~\eqref{eq:general-central-polynomial-bound}.
Only polynomial dependence is used below; no optimized convexity exponent is needed.
\end{proof}

\begin{lemma}[Uniform conductor-dependent edge approximation]
\label{lem:general-uniform-edge-approximation}
Fix $\mu\ge0$.  Uniformly over all normalized primitive newforms of weight $k$,
level $N$, and nebentypus $\chi$, one has, for all sufficiently large $k$,
\begin{align}
 c_0&\ne0,
 \label{eq:general-c0-nonzero}\\
 |\arg c_0|&\ll_\mu 2^{-k/3},
 \label{eq:general-c0-phase}\\
 \sup_{|w|\le1}
 \left|\frac{q_{f,\mu,k,N}(w)}{c_0}-e^{a_Nw}\right|
 &\ll_\mu \frac{a_N}{k}.
 \label{eq:general-uniform-edge-approximation}
\end{align}
For sufficiently large $k$, the number $c_0$ lies in the open right half-plane;
throughout, $\arg c_0$ denotes its principal argument.
\end{lemma}

\begin{proof}
Write $\lambda_f(n):=a_f(n)n^{-(k-1)/2}$, so that
$|\lambda_f(n)|\le d(n)$.  Put
\[
 g_j(s):=\frac{G_N^{(j)}(s)}{G_N(s)},
 \qquad
 \lambda_N(s):=\psi(s)-\log a_N,
 \qquad
 \lambda_0:=\lambda_N(k-1).
\]
For sufficiently large $k$, uniformly in $N\ge1$,
\begin{equation}
 \lambda_0\asymp 1+\log(k\sqrt N),
 \qquad
 \lambda_0\gg\log k.
 \label{eq:general-lambda0-size}
\end{equation}

\medskip
\noindent\emph{Step 1: Uniform bounds away from the middle.}
If $0\le r\le J-1$ and $s_r=k-1-r$, then
\[
 s_r-\frac{k-1}{2}\ge\frac32
\]
when $k$ is even, and the left side is at least $2$ when $k$ is odd.  Hence, for
$0\le b\le\mu$,
\begin{equation}
 |L^{(b)}(f,s_r)|
 \le\sum_{n\ge1}\frac{d(n)(\log n)^b}
 {n^{s_r-(k-1)/2}}
 \ll_b1.
 \label{eq:general-L-right-bound}
\end{equation}
The complete Bell-polynomial formula in
$\lambda_N(s),\psi'(s),\ldots,\psi^{(j-1)}(s)$ gives
\begin{align}
 |g_j(s_r)|&\ll_j(1+\lambda_0)^j,
 \label{eq:general-gamma-quotient-bound}\\
 g_j(s)&=\lambda_N(s)^j
 +O_j\!\left(\frac{(1+\lambda_N(s))^{j-2}}{s}\right)
 \quad(j\ge2),
 \label{eq:general-gamma-leading}
\end{align}
with $g_0=1$ and $g_1=\lambda_N$.  Since
\begin{equation}
 A_{f,\mu}(s)
 =\sum_{b=0}^{\mu}\binom\mu b
 g_{\mu-b}(s)L^{(b)}(f,s),
 \label{eq:general-A-expansion}
\end{equation}
we obtain
\begin{equation}
 |A_{f,\mu}(s_r)|\ll_\mu(1+\lambda_0)^\mu
 \qquad(0\le r\le J-1).
 \label{eq:general-A-right-bound}
\end{equation}

\medskip
\noindent\emph{Step 2: The endpoint coefficient and its phase.}
At $s=k-1$, absolute convergence and Deligne's bound give
\begin{equation}
 L(f,k-1)=1+O(2^{-k/3}),
 \qquad
 L^{(b)}(f,k-1)=O_b(2^{-k/3})\quad(b\ge1),
 \label{eq:general-L-endpoint}
\end{equation}
uniformly in the conductor.  Equations~\eqref{eq:general-gamma-leading} and
\eqref{eq:general-A-expansion} therefore imply
\begin{equation}
 c_0=A_{f,\mu}(k-1)
 =\lambda_0^\mu
 \left(1+O_\mu\!\left(
 \frac1{k(1+\lambda_0)^2}+2^{-k/3}\right)\right),
 \label{eq:general-c0-asymptotic}
\end{equation}
with the evident interpretation $1+O(2^{-k/3})$ for $\mu=0$.
More precisely,
\begin{equation}
 c_0=g_\mu(k-1)
 +O_\mu\bigl((1+\lambda_0)^\mu2^{-k/3}\bigr),
 \label{eq:general-c0-real-main}
\end{equation}
and $g_\mu(k-1)$ is positive for all sufficiently large $k$.  Moreover, the
error term in~\eqref{eq:general-c0-real-main} is
$o_\mu(g_\mu(k-1))$, uniformly in $N$, $\chi$, and $f$.  Thus $c_0$ lies in the
open right half-plane for all sufficiently large $k$, so its principal argument is
well defined and tends uniformly to zero.  This proves
\eqref{eq:general-c0-nonzero}, gives
$|c_0|\asymp_\mu(1+\lambda_0)^\mu$, and proves the phase estimate
\eqref{eq:general-c0-phase}.  Combining this with
\eqref{eq:general-A-right-bound} yields
\begin{equation}
 \left|\frac{A_{f,\mu}(k-1-r)}{c_0}\right|\ll_\mu1,
 \qquad0\le r\le J-1.
 \label{eq:general-ratio-global-bound}
\end{equation}

\medskip
\noindent\emph{Step 3: Near-edge coefficient ratios.}
Let $R=\lfloor3\log k\rfloor$.  For $0\le r\le R$, the normalized
Dirichlet-series exponent satisfies
\[
 s_r-\frac{k-1}{2}\ge\frac k3
\]
for all sufficiently large $k$.  Hence
\begin{equation}
 L(f,s_r)=1+O(2^{-k/3}),
 \qquad
 L^{(b)}(f,s_r)=O_b(2^{-k/3})\quad(b\ge1).
 \label{eq:general-L-near-edge}
\end{equation}
Taylor's formula and the standard polygamma bounds give
\begin{equation}
 \lambda_N(s_r)
 =\lambda_0-r\psi'(k-1)+O(r^2/k^2).
 \label{eq:general-lambda-near-edge}
\end{equation}
Using~\eqref{eq:general-A-expansion},
\eqref{eq:general-gamma-leading}, and
\eqref{eq:general-L-near-edge}, and noting that
$\lambda_N(s_r)\asymp\lambda_0$ in this range, we first obtain the factorized
estimate
\begin{equation}
 \frac{A_{f,\mu}(k-1-r)}{c_0}
 =\left(\frac{\lambda_N(k-1-r)}{\lambda_0}\right)^\mu
 \left(1+O_\mu\!\left(
 \frac1{k(1+\lambda_0)^2}+2^{-k/3}\right)\right),
 \label{eq:general-ratio-edge-factorized}
\end{equation}
where the first factor is interpreted as $1$ when $\mu=0$.
Now~\eqref{eq:general-lambda-near-edge} and
$\psi'(k-1)\asymp k^{-1}$ give
\begin{equation}
 \frac{A_{f,\mu}(k-1-r)}{c_0}
 =1+O_\mu\!\left(
 \frac{r}{k(1+\lambda_0)}
 +\frac1{k(1+\lambda_0)^2}
 +2^{-k/3}\right)
 =1+O_\mu(k^{-1})
 \label{eq:general-ratio-edge-asymptotic}
\end{equation}
uniformly for $0\le r\le R$.  At $r=0$ the ratio is, of course, exactly one.

\medskip
\noindent\emph{Step 4: The last left-half coefficient.}
The index $J$ corresponds to
\[
 s_J=k-1-J\in\left\{\frac k2,\frac{k+1}{2}\right\}.
\]
Lemma~\ref{lem:general-central-polynomial-bound} and
\eqref{eq:general-c0-asymptotic} give, for some $C_\mu$,
\[
 \left|\frac{c_J}{c_0}\right|
 \ll_\mu (Nk)^{C_\mu}\frac{a_N^J}{J!}.
\]
Once $k$ is large enough that $J\ge2C_\mu+1$, the identity
$a_N=2\pi N^{-1/2}$ gives
\[
 N^{C_\mu}a_N^J
 =(2\pi)^JN^{C_\mu-J/2}
 \le a_N(2\pi)^{J-1}.
\]
Stirling's formula therefore gives, for every fixed $B>0$,
\begin{equation}
 \left|\frac{c_J}{c_0}\right|
 \ll_{\mu,B}a_Nk^{-B}.
 \label{eq:general-middle-negligible}
\end{equation}
This covers both the half-weighted middle term when $d$ is even and the final
left-half term when $d$ is odd.

\medskip
\noindent\emph{Step 5: Completion of the edge approximation.}
For $|w|\le1$, split the coefficient sum at $R$.  The term $r=0$ has no error;
therefore~\eqref{eq:general-ratio-edge-asymptotic} and $0<a_N\le2\pi$ give
\[
 \sum_{r=1}^{R}
 \left|\frac{A_{f,\mu}(k-1-r)}{c_0}-1\right|
 \frac{a_N^r}{r!}
 \ll_\mu\frac{a_N}{k}.
\]
By~\eqref{eq:general-ratio-global-bound}, the remaining coefficient tail is
bounded by a constant times $\sum_{r>R}a_N^r/r!$.  Uniformly for
$0<a_N\le2\pi$, this is $O_B(a_Nk^{-B})$ for every fixed $B$.  The same estimate
holds for the tail of $e^{a_Nw}$, and~\eqref{eq:general-middle-negligible} handles
the last coefficient.  This proves~\eqref{eq:general-uniform-edge-approximation}.
\end{proof}

The exact self-inversive phase differs slightly from the root-number phase in
Theorem~\ref{thm:intro-quantitative}.  Write
\begin{equation}
 c_0=|c_0|e^{i\beta_{f,\mu}},
 \qquad
 \gamma_{f,\mu}:=\Gamma_{f,\mu}-\beta_{f,\mu}.
 \label{eq:general-exact-phase-shift}
\end{equation}
Then~\eqref{eq:general-c0-phase} gives
\begin{equation}
 |\beta_{f,\mu}|\ll_\mu2^{-k/3},
 \label{eq:general-beta-small}
\end{equation}
and
\begin{equation}
 e^{2i\gamma_{f,\mu}}
 =\eta_{f,\mu}\frac{\overline{c_0}}{c_0}.
 \label{eq:general-exact-self-inversive-phase}
\end{equation}
Define the exact phase-corrected model
\begin{equation}
 \widetilde\Phi_{N,f,\mu}(\theta)
 :=a_N\sin\theta-D\theta-\gamma_{f,\mu},
 \label{eq:general-corrected-phase}
\end{equation}
and let $\widetilde{\mathcal A}_{N,f,\mu}$ be its zero set
$\cos\widetilde\Phi_{N,f,\mu}=0$ modulo $2\pi$.

\begin{theorem}[Quantitative localization and spacing]
\label{thm:quantitative-localization}
Fix $\mu\ge0$.  There exists $k_0(\mu)$ such that, for every normalized primitive
newform $f\in S_k^{\mathrm{new}}(\Gamma_0(N),\chi)$ with $k\ge k_0(\mu)$, the
zeros $e^{i\theta_j}$ of $\mathcal R_{f,\mu,k,N}$ and the corrected model angles
$\widetilde\alpha_j\in\widetilde{\mathcal A}_{N,f,\mu}$ can be cyclically labeled
so that
\begin{equation}
 \max_j
 \operatorname{dist}_{\R/2\pi\mathbb Z}
 (\theta_j,\widetilde\alpha_j)
 \ll_\mu\frac{a_N}{k^2}.
 \label{eq:general-root-localization-corrected}
\end{equation}
If $\alpha_j\in\mathcal A_{N,f,\mu}$ are the root-number model angles from
\eqref{eq:intro-general-model-set}, then the same labeling satisfies
\begin{equation}
 \max_j
 \operatorname{dist}_{\R/2\pi\mathbb Z}
 (\theta_j,\alpha_j)
 \ll_\mu\frac{a_N}{k^2}+\frac{2^{-k/3}}{k}.
 \label{eq:general-root-localization}
\end{equation}
Consequently, in cyclic order,
\begin{equation}
 \theta_{j+1}-\theta_j
 =\frac{2\pi}{k-2}
 +O_\mu\!\left(\frac{a_N}{k^2}\right),
 \label{eq:general-root-spacing}
\end{equation}
where the final gap is interpreted modulo $2\pi$.  The zeros of
$Q_{f,\mu}^{\mathrm{norm}}$ are obtained by rotating these angles by $-\pi/2$.
\end{theorem}

\begin{proof}
Put $p(w):=q_{f,\mu,k,N}(w)/c_0$ and
\[
 \delta_{f,\mu,k,N}
 :=\sup_{|w|\le1}|p(w)-e^{a_Nw}|.
\]
Lemma~\ref{lem:general-uniform-edge-approximation} gives
\begin{equation}
 \delta_{f,\mu,k,N}\ll_\mu\frac{a_N}{k}.
 \label{eq:general-delta-bound}
\end{equation}
For $w=e^{i\theta}$,~\eqref{eq:general-edge-decomposition} and
\eqref{eq:general-exact-self-inversive-phase} yield
\begin{align}
 e^{-iD\theta}\frac{\mathcal R_{f,\mu,k,N}(e^{i\theta})}{c_0}
 &=Y(\theta)+e^{2i\gamma_{f,\mu}}\overline{Y(\theta)},
 \label{eq:general-real-trig-decomposition}\\
 Y(\theta)&:=e^{-iD\theta}p(e^{i\theta}).
 \notag
\end{align}
Thus, on the universal cover of the circle,
\begin{equation}
 H(\theta):=\frac12e^{-i\gamma_{f,\mu}}
 e^{-iD\theta}
 \frac{\mathcal R_{f,\mu,k,N}(e^{i\theta})}{c_0}
 =\re\bigl(e^{-i\gamma_{f,\mu}}Y(\theta)\bigr)
 \label{eq:general-real-function}
\end{equation}
is real.  When $d$ is odd it is antiperiodic rather than periodic, but its zero set
is $2\pi$-periodic.  Its model is
\begin{equation}
 H_0(\theta)
 =e^{a_N\cos\theta}
  \cos\widetilde\Phi_{N,f,\mu}(\theta),
 \qquad
 \|H-H_0\|_\infty\le\delta_{f,\mu,k,N}.
 \label{eq:general-model-real-function}
\end{equation}

For large $k$ one has $D>a_N$, and
\[
 \widetilde\Phi_{N,f,\mu}'(\theta)
 =a_N\cos\theta-D<0,
 \qquad
 \widetilde\Phi_{N,f,\mu}(\theta+2\pi)
 =\widetilde\Phi_{N,f,\mu}(\theta)-d\pi.
\]
Choose $b_0\in\R$ with
$\widetilde\Phi_{N,f,\mu}(b_0)\in\pi\mathbb Z$, and let $b_j$ be determined by
\[
 \widetilde\Phi_{N,f,\mu}(b_j)
 =\widetilde\Phi_{N,f,\mu}(b_0)-j\pi,
 \qquad0\le j\le d.
\]
Then $b_d=b_0+2\pi$.  At the points $b_j$, the model in
\eqref{eq:general-model-real-function} has alternating signs and absolute value at
least $e^{-a_N}\ge e^{-2\pi}$.  By~\eqref{eq:general-delta-bound}, the exact
function has the same signs for large $k$.  Hence it has a zero in each interval
$(b_j,b_{j+1})$.  These sign changes give $d$ distinct zeros of
$\mathcal R_{f,\mu,k,N}$ on the unit circle.  Since
$c_d=\eta_{f,\mu}\overline{c_0}\ne0$, the polynomial has degree $d$; hence these
zeros exhaust all zeros of $\mathcal R_{f,\mu,k,N}$, counted with multiplicity.  In
particular every interval contains exactly one zero, and all zeros are simple.

The same interval contains exactly one corrected model angle
$\widetilde\alpha_j$.  At the corresponding exact zero $\theta_j$,
\eqref{eq:general-model-real-function} gives
\[
 |\cos\widetilde\Phi_{N,f,\mu}(\theta_j)|
 \le e^{a_N}\delta_{f,\mu,k,N}.
\]
Since
$|\widetilde\Phi_{N,f,\mu}'|\ge D-a_N$, it follows that
\begin{equation}
 \operatorname{dist}_{\R/2\pi\mathbb Z}
 (\theta_j,\widetilde\alpha_j)
 \le
 \frac{\arcsin(e^{a_N}\delta_{f,\mu,k,N})}{D-a_N}
 \ll_\mu\frac{a_N}{k^2},
 \label{eq:general-localization-explicit}
\end{equation}
which proves~\eqref{eq:general-root-localization-corrected}.

The corrected and root-number phases differ by the constant
$\beta_{f,\mu}$.  Applying the mean-value theorem and
\eqref{eq:general-beta-small} gives
\[
 \operatorname{dist}_{\R/2\pi\mathbb Z}
 (\widetilde\alpha_j,\alpha_j)
 \le\frac{|\beta_{f,\mu}|}{D-a_N}
 \ll_\mu\frac{2^{-k/3}}{k}.
\]
This proves~\eqref{eq:general-root-localization}.

Finally, adjacent corrected model angles differ in phase by $\pi$.  For a suitable
intermediate point $\xi_j$,
\begin{equation}
 \widetilde\alpha_{j+1}-\widetilde\alpha_j
 =\frac{\pi}{D-a_N\cos\xi_j}
 =\frac{2\pi}{d}+O\!\left(\frac{a_N}{k^2}\right),
 \label{eq:general-model-spacing}
\end{equation}
cyclically at the final gap.  Combining this with the corrected localization
\eqref{eq:general-root-localization-corrected} proves
\eqref{eq:general-root-spacing}.
Equation~\eqref{eq:general-Q-R-rotation} gives the final rotation statement.
\end{proof}

\begin{corollary}[Quantitative half-step separation]
\label{cor:quantitative-half-step}
Fix $\mu\ge0$.  For all sufficiently large $k$, uniformly over all normalized
primitive newforms $f\in S_k^{\mathrm{new}}(\Gamma_0(N),\chi)$, merge the roots of
$Q_{f,\mu}^{\mathrm{norm}}$ and $Q_{f,\mu+1}^{\mathrm{norm}}$ in cyclic order.
The two derivative orders alternate, and every gap between adjacent roots in the
merged set is
\begin{equation}
 \frac{\pi}{k-2}
 +O_\mu\!\left(\frac{a_N}{k^2}+\frac{2^{-k/3}}{k}\right)
 =\frac{\pi}{k-2}+O_\mu(k^{-2}).
 \label{eq:general-quantitative-half-step}
\end{equation}
The same angular statement holds for the conductor-rescaled polynomials
$R_{f,\mu}^{(N)}$ on $|X|=N^{-1/2}$.
\end{corollary}

\begin{proof}
By definition,
\[
 \Gamma_{f,\mu+1}=\Gamma_{f,\mu}+\frac\pi2.
\]
Hence the union of the two root-number model sets consists of consecutive phase
levels separated by $\pi/2$.  The mean-value theorem gives every adjacent model gap
as
\[
 \frac{\pi/2}{D-a_N\cos\xi}
 =\frac{\pi}{d}+O\!\left(\frac{a_N}{k^2}\right).
\]
The minimum gap in the union of the two model sets is $\gg k^{-1}$, whereas
\eqref{eq:general-root-localization}, applied to orders $\mu$ and $\mu+1$, gives a
localization radius $o(k^{-1})$.  Hence, for all sufficiently large $k$, that
radius is less than one quarter of the minimum model gap.  Nearest-model matching
is therefore unique and preserves the merged cyclic order, so the two derivative
orders alternate.  The same localization estimate then gives
\eqref{eq:general-quantitative-half-step}.
The reciprocal rotation and radial conductor rescaling do not change angular gaps.
\end{proof}

\begin{remark}
For $\mu=0$, the phase in~\eqref{eq:intro-general-phase}, after the reciprocal
rotation in~\eqref{eq:general-Q-R-rotation}, is equivalent to the classical
conductor-dependent model used in~\cite{JinMaOnoSound}.  In the arbitrary-nebentypus case, the result of Liu--Park--Song \cite{LiuParkSong} allows finitely many possible exceptional newforms.  Those papers obtain sharper localization in the
classical setting.  The estimates here are organized instead to be uniform in the
level and nebentypus for every fixed derivative order and to retain enough phase
information to compare consecutive orders.
\end{remark}

\section*{Use of AI-assisted tools}

For transparency, during the preparation of this manuscript the author used OpenAI
ChatGPT to discuss possible gaps in arguments, improve exposition, and edit LaTeX and
English prose.  The tool was not used as a source of mathematical results or as a formal proof
verifier.  All proofs, computations, statements, and references were independently
checked by the author, who takes full responsibility for the manuscript.

\begingroup
\setlength{\emergencystretch}{3em}

\endgroup


\begin{thebibliography}{99}

\bibitem{BabeiRolenWagner}
A.~Babei, L.~Rolen, and I.~Wagner,
\emph{The Riemann hypothesis for period polynomials of Hilbert modular forms},
J. Number Theory \textbf{218} (2021), 44--61.

\bibitem{BorceaBranden}
J.~Borcea and P.~Br\"and\'en,
\emph{P\'olya--Schur master theorems for circular domains and their boundaries},
Ann. of Math. (2) \textbf{170} (2009), no.~1, 465--492.

\bibitem{BrandenChasse}
P.~Br\"and\'en and M.~Chasse,
\emph{Classification theorems for operators preserving zeros in a strip},
J. Anal. Math. \textbf{132} (2017), no.~1, 177--215.

\bibitem{BrelandEtAl}
L.~Breland, K.~H. Le, J.~Ni, L.~O'Brien, H.~Xue, and D.~Zhu,
\emph{Interlacing of zeros of period polynomials},
J. Math. Soc. Japan \textbf{77} (2025), no.~1, 255--299.

\bibitem{ConreyFarmerImamoglu}
J.~B. Conrey, D.~W. Farmer, and \"O.~\.{I}mamo\u{g}lu,
\emph{The nontrivial zeros of period polynomials of modular forms lie on the unit circle},
Int. Math. Res. Not. IMRN (2013), no.~20, 4758--4771.

\bibitem{CravenCsordas}
T.~Craven and G.~Csordas,
\emph{Composition theorems, multiplier sequences and complex zero decreasing sequences},
in \emph{Value Distribution Theory and Related Topics}, Adv. Complex Anal. Appl.,
vol.~3, Kluwer Academic Publishers, Boston, MA, 2004, 131--166.

\bibitem{deBruijnTrig}
N.~G. de Bruijn,
\emph{The roots of trigonometric integrals},
Duke Math. J. \textbf{17} (1950), no.~3, 197--226.

\bibitem{Deligne}
P.~Deligne,
\emph{La conjecture de Weil. I},
Inst. Hautes \'Etudes Sci. Publ. Math. No.~\textbf{43} (1974), 273--307.

\bibitem{DiamantisRolen}
N.~Diamantis and L.~Rolen,
\emph{Eichler cohomology and zeros of polynomials associated to derivatives of $L$-functions},
J. Reine Angew. Math. \textbf{770} (2021), 1--25.

\bibitem{DiamantisRolenSurvey}
N.~Diamantis and L.~Rolen,
\emph{Period polynomials, derivatives of $L$-functions, and zeros of polynomials},
Res. Math. Sci. \textbf{5} (2018), no.~1, Paper No.~9, 15~pp.

\bibitem{ElGuindyRaji}
A.~El-Guindy and W.~Raji,
\emph{Unimodularity of zeros of period polynomials of Hecke eigenforms},
Bull. Lond. Math. Soc. \textbf{46} (2014), no.~3, 528--536.

\bibitem{Grace}
J.~H. Grace,
\emph{The zeros of a polynomial},
Proc. Cambridge Philos. Soc. \textbf{11} (1902), 352--357.

\bibitem{ImKim}
B.-H. Im and H.~Kim,
\emph{Riemann hypothesis for period polynomials attached to the derivatives of $L$-functions of cusp forms for $\Gamma_0(N)$},
J. Math. Anal. Appl. \textbf{509} (2022), no.~2, Paper No.~125971.

\bibitem{IwaniecKowalski}
H.~Iwaniec and E.~Kowalski,
\emph{Analytic Number Theory},
American Mathematical Society Colloquium Publications, vol.~53,
American Mathematical Society, Providence, RI, 2004.

\bibitem{JinMaOnoSound}
S.~Jin, W.~Ma, K.~Ono, and K.~Soundararajan,
\emph{The Riemann hypothesis for period polynomials of modular forms},
Proc. Natl. Acad. Sci. USA \textbf{113} (2016), no.~10, 2603--2608.

\bibitem{KatkovaTyaglovVishnyakova}
O.~Katkova, M.~Tyaglov, and A.~Vishnyakova,
\emph{Hermite--Poulain theorems for linear finite difference operators},
Constr. Approx. \textbf{52} (2020), no.~3, 357--393.

\bibitem{KabluchkoMultiplicative}
Z.~Kabluchko,
\emph{Zero distribution of multiplicative Hermite and Laguerre polynomials},
arXiv:2511.01456 [math.PR] (2025), version~2 (2026), 25~pp.

\bibitem{Levin}
B.~Ya. Levin,
\emph{Distribution of Zeros of Entire Functions}, revised ed.,
Translations of Mathematical Monographs, vol.~5, American Mathematical Society,
Providence, RI, 1980.

\bibitem{LiuParkSong}
Y.~P. Liu, P.~S. Park, and Z.~Q. Song,
\emph{The ``Riemann hypothesis'' is true for period polynomials of almost all newforms},
Res. Math. Sci. \textbf{3} (2016), Paper No.~31, 18~pp.

\bibitem{MarcusSpielmanSrivastavaFiniteFree}
A.~W. Marcus, D.~A. Spielman, and N.~Srivastava,
\emph{Finite free convolutions of polynomials},
Probab. Theory Related Fields \textbf{182} (2022), no.~3--4, 807--848.

\bibitem{Miyake}
T.~Miyake,
\emph{Modular Forms},
Springer Monographs in Mathematics, Springer-Verlag, Berlin, 2006.

\bibitem{Obreschkoff}
N.~Obreschkoff,
\emph{Verteilung und Berechnung der Nullstellen reeller Polynome},
Hochschulb\"ucher f\"ur Mathematik, vol.~55, VEB Deutscher Verlag der
Wissenschaften, Berlin, 1963.

\bibitem{PolyaSchur}
G.~P\'olya and I.~Schur,
\emph{\"Uber zwei Arten von Faktorenfolgen in der Theorie der algebraischen Gleichungen},
J. Reine Angew. Math. \textbf{144} (1914), 89--113.

\bibitem{RahmanSchmeisser}
Q.~I. Rahman and G.~Schmeisser,
\emph{Analytic Theory of Polynomials: Critical Points, Zeros and Extremal Properties},
London Mathematical Society Monographs New Series, vol.~26, Oxford University Press, Oxford, 2002.

\bibitem{Szego}
G.~Szeg\H{o},
\emph{Bemerkungen zu einem Satz von J.~H. Grace \"uber die Wurzeln algebraischer Gleichungen},
Math. Z. \textbf{13} (1922), 28--55.

\end{thebibliography}
\end{document}